\documentclass[11pt,english]{article}
\usepackage[T1]{fontenc}
\usepackage{textcomp}
\usepackage[latin9]{inputenc}
\usepackage[a4paper]{geometry}
\usepackage{parskip}
\usepackage{xcolor}
\usepackage{babel}
\usepackage{calc}
\usepackage{units}
\usepackage{mathrsfs}
\usepackage{mathtools}
\usepackage{enumitem}
\usepackage{amsmath}
\usepackage{amsthm}
\usepackage{amssymb}
\usepackage[pdfusetitle,
 bookmarks=true,bookmarksnumbered=false,bookmarksopen=false,
 breaklinks=false,pdfborder={0 0 0},pdfborderstyle={},backref=false,colorlinks=true]
 {hyperref}
\hypersetup{
 linkcolor=blue,  citecolor=blue, urlcolor=blue}

\makeatletter

\providecommand{\tabularnewline}{\\}

\numberwithin{equation}{section}
\numberwithin{figure}{section}
\numberwithin{table}{section}
\theoremstyle{plain}
\newtheorem{thm}{\protect\theoremname}[section]
\theoremstyle{definition}
\newtheorem{defn}[thm]{\protect\definitionname}
\theoremstyle{remark}
\newtheorem{rem}[thm]{\protect\remarkname}
\theoremstyle{plain}
\newtheorem*{thm*}{\protect\theoremname}
\theoremstyle{remark}
\newtheorem*{acknowledgement*}{\protect\acknowledgementname}
\theoremstyle{plain}
\newtheorem{fact}[thm]{\protect\factname}
\theoremstyle{plain}
\newtheorem{prop}[thm]{\protect\propositionname}
\theoremstyle{plain}
\newtheorem{lem}[thm]{\protect\lemmaname}
\theoremstyle{plain}
\newtheorem{cor}[thm]{\protect\corollaryname}
\theoremstyle{remark}
\newtheorem{claim}[thm]{\protect\claimname}
\theoremstyle{plain}
\newtheorem{conjecture}[thm]{\protect\conjecturename}

\usepackage{bbm}
\usepackage[abbrev]{amsrefs}
\usepackage{xypic}
\usepackage{stackrel,amssymb,amsmath}

\newcommand{\rleftarrows}{\mathrel{\raise.75ex\hbox{\oalign{%
  $\hfil\scriptstyle\relbar$\cr
  \vrule width0pt height.5ex$\scriptstyle\smash\leftarrow$\cr}}}}
\newcommand{\rightlarrows}{\mathrel{\raise.75ex\hbox{\oalign{%
  $\scriptstyle\rightarrow$\hfil\cr
  $\scriptstyle\vrule width0pt height.5ex\relbar$\cr}}}}
\newcommand{\Rrelbar}{\mathrel{\raise.75ex\hbox{\oalign{%
  $\scriptstyle\relbar$\cr
  \vrule width0pt height.5ex$\scriptstyle\relbar$}}}}

\makeatletter
\def\rightleftarrowsfill@{\arrowfill@\rleftarrows\Rrelbar\rightlarrows}
\newcommand{\xrightleftarrows}[2][]{\ext@arrow 3399\rightleftarrowsfill@{#1}{#2}}
\makeatother

\setlist[enumerate]{itemsep=-0.5\parsep, topsep=-0.5\parsep}

\DeclareMathOperator{\rank}{rank}

\DeclareMathOperator{\Aut}{Aut}

\DeclareMathOperator{\Spec}{Spec}

\makeatother

\providecommand{\acknowledgementname}{Acknowledgement}
\providecommand{\claimname}{Claim}
\providecommand{\conjecturename}{Conjecture}
\providecommand{\corollaryname}{Corollary}
\providecommand{\definitionname}{Definition}
\providecommand{\factname}{Fact}
\providecommand{\lemmaname}{Lemma}
\providecommand{\propositionname}{Proposition}
\providecommand{\remarkname}{Remark}
\providecommand{\theoremname}{Theorem}

\begin{document}
\title{Geometry over finite local rings: Rigidity and Isospectrality}
\author{Yishai Lavi and Ori Parzanchevski}
\maketitle
\begin{abstract}
We study the simplicial order complexes obtained from free modules
over finite local rings. These complexes arise naturally as geodesic
spheres in Bruhat-Tits buildings over non-archimedean local fields.
We establish two forms of rigidity, showing that their automorphism
groups arise from the underlying algebraic group, and that they are
determined by sparse induced subgraphs. We compute the spectra of
these subgraphs and show that they form excellent expanders, which
results in expansion for geodesic powers of Bruhat-Tits buildings.
The computation also reveals that local rings with the same residue
order give rise to isospectral induced subgraphs. Combining this with
our rigidity results we show that the graphs arising from $n$-spaces
over $\mathbb{Z}/p^{r}$ and $\mathbb{F}_{p}[t]/(t^{r})$ are isospectral
and non-isomorphic.
\end{abstract}

\section{Introduction}

The spherical building associated with the group $PGL_{d}(\mathbb{F}_{q})$
is the simplicial complex of nontrivial subspaces of $\mathbb{F}_{q}^{d}$,
with cells given by flags. These complexes have been studied for almost
two centuries, and appear naturally in many areas of combinatorics
and algebra. One such appearance is as the local vista of cells in
(affine) Bruhat-Tits buildings associated with $p$-adic groups; for
example, the link of a vertex in the affine building of $PGL_{d}(\mathbb{Q}_{p})$
is the spherical building of $PGL_{d}(\mathbb{F}_{p})$. 

In this paper we study building-like complexes arising from $PGL_{d}$
over finite local rings such as $\mathbb{Z}/p^{r}$, which we call
\emph{free projective spaces}. These appear naturally when considering
``geodesic'' spheres of radius $r$ in $p$-adic buildings. Our
main results are two types of rigidity theorems for these complexes,
and a spectral analysis of bipartite subgraphs which they induce. 

The first rigidity result determines that all automorphisms of our
complexes are algebraic in nature, as is known to be true for classical
buildings of high rank. The second result shows that the combinatorial
information contained in very sparse subgraphs of the complex is enough
to reconstruct the full complex. To motivate our spectral analysis,
let us explain our interest in geodesic spheres, which arises from
a powering operation for simplicial complexes suggested in \cite{Kaufman2019Freeflagsover}.
It is a simple observation that spherical buildings over finite fields
(the $r=1$ case) are excellent \emph{expanders}, namely, their nontrivial
eigenvalues are small in magnitude comparing to the trivial one (the
Perron-Frobenius). This has surprisingly deep applications: in \cite{Gar73},
Garland regards the expansion of links in the building of $PGL_{d}(\mathbb{Q}_{p})$
as a $p$-adic form of curvature, and develops a local-to-global technique
to establish vanishing of cohomology for finite quotients of the building,
resolving a question of Serre on $S$-arithmetic groups (see also
\cite{grinbaum2022curvature}). Recently, Garland's technique was
rediscovered and extended in the context of high-dimenional expanders,
yielding new applications in combinatorics and computer science \cite{Oppenheim2017Localspectral,Kaufman2017Highorderrandom,Abdolazimi2022MatrixTrickleTheorem}. 

The study of geodesic spheres in \cite{Kaufman2019Freeflagsover}
is motivated by the search for a powering operation for simplicial
complexes, which preserves local, or high-dimensional expansion. It
is shown that the ``natural'' spheres in the buildings are not good
expanders \cite[§4]{Kaufman2019Freeflagsover}, and the notion of
geodesic spheres (see Remark \ref{rem:geodes} for their definition)
is suggested to overcome this. Our spectral analysis (Theorem \ref{thm:Spectrum})
shows that the geodesic spheres are good expanders, and this implies
that the geodesic powering operation preserves high-dimenional expansion.
This extends the main result of \cite{Kaufman2019Freeflagsover} which
treats only $PGL_{3}$, and in addition gives a much better proof
by exploring the structure of the complexes in greater details.

It is especially interesting to compare the geodesic spheres in the
buildings of $PGL_{d}$ over $\mathbb{Q}_{p}$, and over the field
of Laurent series $\mathbb{F}_{p}((t))$. The similarities and differences
between these fields and their buildings underpin deep mathematical
theories, e.g.\ Deligne--Kazhdan's ``close fields'', motivic uniformity,
and perfectoid theory. From our perspective, in both buildings the
vertex link (which is the geodesic sphere of radius $r=1$) is the
spherical building of $PGL_{d}(\mathbb{F}_{p})$, but when moving
to larger spheres we find ourselves comparing free projective spaces
over $\mathbb{Z}/(p^{r})$ and $\mathbb{F}_{p}[t]/(t^{r})$ respectively.
Our spectral analysis shows that that graphs arising from spheres
in the buildings of $PGL_{p}(\mathbb{Q}_{q})$ and $PGL_{d}(\mathbb{F}_{p}((t)))$
are isospectral, whereas our rigidity results show that they are not
isomorphic (Theorem \ref{thm:iso-noniso}), resolving a conjecture
stated in \cite{Kaufman2019Freeflagsover}. For a survey of the long
and interesting history of the problem of graph isospectrality see
\cite{gordon2005isospectral}.

\vspace{.5\baselineskip}

We now move to more precise terms. Let $\mathcal{O}$ be a Dedekind
domain, $\mathfrak{p}\trianglelefteq\mathcal{O}$ a prime ideal of
residue order $q$, and $r\geq1$. We observe 
\[
\mathcal{O}_{r}:=\nicefrac{\mathcal{O}}{\mathfrak{p}^{r}},
\]
which is a finite local ring of size $q^{r}$. 
\begin{defn}
\label{def:The-free-projective}The \emph{free projective $d$-space
over $\mathcal{O}_{r}$, }denoted $\mathbb{P}_{fr}^{d-1}(\mathcal{O}_{r})$,
is the simplicial complex whose vertices are the \textbf{free }submodules
$0\lneq V\lneq\mathcal{O}_{r}^{d}$, and whose cells correspond to
flags of free submodules. 
\end{defn}

\begin{rem}
\label{rem:geodes}Let us hint where this complex comes from: Denote
by $F$ the field of fractions of $\mathcal{O}$, and by $F_{\mathfrak{p}}$
its completion at $\mathfrak{p}$. Let $\mathcal{B}_{d,\mathfrak{p}}$
be the affine Bruhat-Tits building associated with $PGL_{d}(F_{\mathfrak{p}})$
(see e.g.\ \cite{Brown1989,Garrett1997} for definitions). We put
a new pure simplicial structure on the vertices of $\mathcal{B}_{d,\mathfrak{p}}$
as follows: $v_{0},\ldots,v_{d-1}$ form a $d$-cell if for every
$i$, possibly after reordering, there is a geodesic path of length
$r$, composed of edges of color one, from $v_{i}$ to $v_{i+1\,(\mathrm{mod}\,d)}$
(see \cite[§2.3]{Kaufman2019Freeflagsover} for more details). We
call the link of a vertex in this new complex an \emph{$r$-geodesic
sphere}; by \cite[Prop.\ 3.6]{Kaufman2019Freeflagsover}, it is a
complex isomorphic to $\mathbb{P}_{fr}^{d-1}(\mathcal{O}_{r})$.
\end{rem}

We think of $\mathcal{X}=\mathbb{P}_{fr}^{d-1}(\mathcal{O}_{r})$
as a colored complex, where the color of a vertex is the rank of its
associated submodule in $\mathcal{O}_{r}^{d}$, so that $\mathcal{X}$
is a pure $(d-2)$-dimensional complex, and every facet $\sigma\in\mathcal{X}^{d-2}$
has vertices of all colors. Denote by $\Aut^{0}(\mathcal{X})$ the
automorphisms of $\mathcal{X}$ which preserve colors, and by $\Aut\left(\mathcal{X}\right)$
all automorphisms. In Section \ref{sec:Automorphic-ridigity}, we
show that all automorphisms arise from the underlying algebraic group,
as follows: Let $A\in GL_{d}(\mathcal{O}_{r})$ and $\tau\in\Aut_{Ring}(\mathcal{O}_{r})$,
and denote $\overline{\tau}(v)=\left(\tau(v_{1}),\ldots,\tau(v_{d})\right)^{T}$.
The map $\varphi_{A,\tau}(V)=\left\{ A\overline{\tau}\left(v\right)\,\middle|\,v\in V\right\} $
induces an automorphism of $\mathcal{X}$, and so does 
\[
\varphi_{\bot}\left(V\right)=V^{\bot}:=\left\{ w\in\mathcal{O}_{r}^{d}\,\middle|\,\forall v\in V:\left\langle v,w\right\rangle =0\right\} ,
\]
where $\left\langle \ ,\ \right\rangle $ is the standard bilinear
form on $\mathcal{O}_{r}^{d}$.
\begin{thm}
\label{thm:aut0X}Let $\varphi\in\Aut^{0}(\mathcal{X})$. If $d\geq3$,
then there exist $A\in GL_{d}(\mathcal{O}_{r})$ and $\tau\in\Aut_{Ring}(\mathcal{O}_{r})$
such that $\varphi=\varphi_{A,\tau}$, so that 
\begin{equation}
\Aut^{0}\left(\mathcal{X}\right)\cong PGL_{d}\left(\mathcal{O}_{r}\right)\rtimes\Aut_{Ring}(\mathcal{O}_{r}),\label{eq:Aut0}
\end{equation}
and the full automorphism group of $\mathcal{X}$ is:
\begin{equation}
\Aut\left(\mathcal{X}\right)=\Aut^{0}\left(\mathcal{X}\right)\rtimes\left\langle \varphi_{\bot}\right\rangle \cong\left(PGL_{d}\left(\mathcal{O}_{r}\right)\rtimes\Aut_{Ring}(\mathcal{O}_{r})\right)\rtimes\left(\nicefrac{\mathbb{Z}}{2\mathbb{Z}}\right).\label{eq:Aut}
\end{equation}
\end{thm}

The analogue result for spherical buildings associated with classical
groups was proved by Tits \cite[§9]{Tits1974Buildingssphericaltype}
(for the affine case, see \cite{weiss2008structure}). We should therefore
point out that our complexes are not buildings -- see Proposition
\ref{prop:not-building}. Comparing our work with the classic case
where one works over a field, the main difference is that the sum
and intersection of free modules need not be free, in which case they
do not appear in the complex. The proof itself involves a study of
the relations between algebraic and combinatorial incidence aspects
of the submodules. Many of these are rather intricate, and become
simple when reduced to the field case ($r=1$).

Section \ref{sec:Subgraph-rigidity} which follows, establishes two
forms of ``subgraph rigity'':
\begin{thm*}[\ref{thm:subgraph_rigidity}]
Let $\mathcal{X}=\mathbb{P}_{fr}^{d-1}(\mathcal{O}_{r})$ as before,
and let $\mathcal{X}_{m,n}$ be the subgraph induced by the vertices
of colors $m$ and $n$ for some $1\leq m<n\leq d-1$.
\begin{enumerate}
\item The complex $\mathcal{X}$ can be reconstructed from $\mathcal{X}_{m,n}$.
\item Every automorphism of $\mathcal{X}_{m,n}$ extends uniquely to an
automorphism of $\mathcal{X}$, giving
\[
\Aut^{0}\left(\mathcal{X}_{m,n}\right)=\Aut^{0}\left(\mathcal{X}\right),\text{ and }\Aut\left(\mathcal{X}_{m,n}\right)=\begin{cases}
\Aut\left(\mathcal{X}\right) & m+n=d\\
\Aut^{0}\left(\mathcal{X}\right) & else.
\end{cases}
\]
\end{enumerate}
\end{thm*}
The challenge here is to exploit the rather sparse combinatorial information
given by two ranks to reconstruct the entire complex. Along the reconstruction,
we establish that every partial automorphism can be extended (uniquely)
as we reveal more of the original complex. Again, the field case $r=1$
is trivial in comparison with the general one.

Section \ref{sec:Spectrum-of-subgraphs} determines the spectrum of
the graphs $\mathcal{X}_{1,n}$ induced by modules of ranks $1$ and
$2\leq n\leq d-1$ in $\mathcal{X}$. Theorem \ref{thm:Spectrum}
gives the complete spectrum with multiplicities, and in particular
we obtain that the second normalized eigenvalue equals 
\[
\sqrt{\frac{\left[d-n\right]_{q}}{\left[n\right]_{q}\left[d-1\right]_{q}}q^{n-1}}\approx\frac{1}{\sqrt{q^{n-1}}},
\]
where $\left[n\right]_{q}=\frac{q^{n}-1}{q-1}$ is the $q$-number.
In particular, this is independent of $r$, and can be made arbitrarily
small by changing $q$ or $n$. This shows that these graphs are excellent
expanders, which was the motivation to exploring $\mathcal{X}_{1,2}$
for $\mathcal{X}=\mathbb{P}_{fr}^{2}(\mathcal{O}_{r})$ in \cite{Kaufman2019Freeflagsover}. 

In addition, Theorem \ref{thm:Spectrum} shows that the spectrum of
$\mathcal{X}_{1,n}$ depends only on $d,n,r$ and the residue order
$q$. Thus, graphs arising from rings such as $\mathbb{Z}/p^{r}$
and $\mathbb{F}_{p}[t]/(t^{r})$ are isospectral, and it is natural
to ask whether they are isomorphic. To prove that they are not, we
use both rigidity theorems: as $\mathcal{X}_{1,n}$ has the same automorphism
group as $\mathcal{X}$ (Theorem \ref{thm:subgraph_rigidity}), which
is given by the underlying algebra (Theorem \ref{thm:aut0X}), it
is enough to determine that the corresponding algebraic groups are
not isomorphic. This turns out to hold with the exception of the case
$d-1=q=r=n=2$; in this case the automorphism groups turn out to be
equivalent, but the graphs themselves are still non-isomorphic. We
obtain:
\begin{thm*}[\ref{thm:iso-noniso}]
If $p\in\mathbb{Z}$ is a prime, then for any $r\geq2$, $d\geq3$
and $1<n<d$ the graphs 
\[
\mathbb{P}_{1,n}^{d-1}\left(\mathbb{Z}/(p^{r})\right),\quad\text{and }\quad\mathbb{P}_{1,n}^{d-1}\left(\mathbb{F}_{p}[t]/(t^{r})\right)
\]
are isospectral and non-isomorphic.
\end{thm*}
\begin{acknowledgement*}
This research was supported by Israel Science Foundation (grant No.\ 2990/21).
\end{acknowledgement*}

\section{Preliminaries}

In this section we collect and prove some algebraic and combinatorial
properties of our complexes, and show that they fall short of being
Tits buildings. Throughout the section $\mathfrak{p}$ is a prime
ideal of residue order $q$ in a Dedekind domain $\mathcal{O}$, and
$\mathcal{O}_{r}=\mathcal{O}/\mathfrak{p}^{r}$ for a fixed $r\geq1$.
By some abuse of notation we denote the image of $\mathfrak{p}$ in
$\mathcal{O}_{r}$ by $\mathfrak{p}$ as well, and we fix some uniformiser
$\pi\in\mathfrak{p}\backslash\mathfrak{p}^{2}$.

We shall call a free $\mathcal{O}_{r}$-module of rank $n$ an \emph{$n$-space,
and }a vector $v\in\mathcal{O}_{r}^{d}$ \emph{primitive }if it spans
a $1$-space, which is equivalent to having at least one coordinate
in $\mathcal{O}_{r}^{\times}=\mathcal{O}_{r}\backslash\mathfrak{p}$.
This paper is mostly concerned with $\mathcal{O}_{r}$-submodules
of $\mathcal{O}_{r}^{d}$, and we list now some simple algebraic observations
regarding them, each a simple exercise using the previous ones.
\begin{fact}
\label{fact:fact}Let $V\leq\mathcal{O}_{r}^{d}$, and $n$ the minimal
size of a generating set for $V$.
\begin{enumerate}
\item All the ideals of $\mathcal{O}_{r}$ are $\mathcal{O}_{r}=\mathfrak{p}^{0}>\mathfrak{p}>\ldots>\mathfrak{p}^{r-1}>\mathfrak{p}^{r}=0$.
\item $\alpha\mid\beta$ or $\beta\mid\alpha$ for any $\alpha,\beta\in\mathcal{O}_{r}$.
\item Any $A\in M_{n\times m}(\mathcal{O}_{r})$ can be made triangular
by elementary row operations, and can be brought to Smith Normal Form
(SNF) by elementary row+column operations.
\item $n\leq d$, and if $S$ generates $V$ then there exists $S'\subseteq S$
of size $n$ which generates $V$.
\item \label{enu:equiv-GLd}$V$ is equivalent under $GL_{d}(\mathcal{O}_{r})$
to $\mathcal{O}_{r}^{m}\times\mathfrak{p}^{k_{1}}\times\ldots\times\mathfrak{p}^{k_{n-m}}\times0^{\times(d-n)}$
for some $0<k_{1}\leq\ldots\leq k_{n-m}<r$. We say $V$ is of type
$(m\,;k_{1},\ldots,k_{n-m})$; it is an $n$-space iff $m=n$.
\item For $V$ of type $(m\,;k_{1},\ldots,k_{n-m})$, $\mathcal{O}_{r}^{d}/V$
and $V^{\bot}$ are of type $(d-n\,;r-k_{n-m},\ldots,r-k_{1})$.
\item If $W\leq V$ then $W$ is $n$-generated.
\end{enumerate}
\end{fact}

Let $\mathcal{X}=\mathbb{P}_{fr}^{d-1}\left(\mathcal{O}_{r}\right)$
be as in Definition \ref{def:The-free-projective}, and let $\mathcal{X}_{n}$
be the vertices of color $n$ in $\mathcal{X}$, which correspond
to $n$-spaces in $\mathcal{O}_{d}^{r}$. Recall the $q$-numbers
$[n]_{q}=\frac{q^{n}-1}{q-1}$, and let $[n]_{!q}=[n]_{q}[n-1]_{q}\ldots[1]_{q}$
and ${d \choose n}_{q}=\frac{[d]_{!q}}{[n]_{!q}[d-n]_{!q}}$ denote
the $q$-factorial and $q$-binomial coefficients.
\begin{prop}
\label{prop:Xn}The number of $n$-spaces in $\mathcal{O}_{r}^{d}$,
namely $\left|\mathcal{X}_{n}\right|$, is 
\[
S_{n}^{d}\coloneqq{d \choose n}_{q}\cdot q^{\left(r-1\right)n\left(d-n\right)}.
\]
\end{prop}

\begin{proof}
By Fact \ref{fact:fact}(\ref{enu:equiv-GLd}) $GL_{d}\left(\mathcal{O}_{r}\right)$
acts transitively on $\mathcal{X}_{n}$, and
\begin{align*}
\left|GL_{d}\left(\mathcal{O}_{r}\right)\right| & =\left(\prod\nolimits_{i=0}^{d-1}q^{d}-q^{i}\right)q^{d^{2}\left(r-1\right)},\\
\left|\text{Stab}_{GL_{d}\left(\mathcal{O}_{r}\right)}\left(\left\langle e_{1},\ldots,e_{n}\right\rangle \right)\right| & =\left|\left(\begin{matrix}GL_{n}(\mathcal{O}_{r}) & M_{n\times\left(d-n\right)}(\mathcal{O}_{r})\\
0 & GL_{d-n}(\mathcal{O}_{r})
\end{matrix}\right)\right|\\
 & {\textstyle =\prod_{i=0}^{n-1}\left(q^{n}-q^{i}\right)\prod_{i=0}^{d-1-n}\left(q^{d-n}-q^{i}\right)\cdot q^{rn(d-n)+(r-1)(n^{2}+(d-n)^{2})}},
\end{align*}
give $\left[GL_{d}\left(\mathcal{O}_{r}\right):\text{Stab}_{GL_{d}\left(\mathcal{O}_{r}\right)}\left(\left\langle e_{1},\ldots,e_{n}\right\rangle \right)\right]=S_{n}^{d}$.
\end{proof}
This is in fact a special case of the next proposition:
\begin{prop}
\label{prop:contains_submodules}Let $M\leq\mathcal{O}_{r}^{d}$ be
of type $\left(m\,;k_{1},\ldots,k_{t}\right)$. Then:.
\begin{enumerate}
\item The number of $n$-spaces \textbf{contained }in $M$ is:
\[
S_{n}^{m}\cdot q^{n\left(\sum_{j=1}^{t}r-k_{j}\right)}
\]
\item The number of $n$-spaces (in $\mathcal{O}_{r}^{d}$) \textbf{containing}
$M$ is:
\[
S_{d-n}^{d-m-t}\cdot q^{\left(d-n\right)\left(\sum_{j=1}^{t}k_{j}\right)}=S_{n-m-t}^{d-m-t}\cdot q^{\left(d-n\right)\left(\sum_{j=1}^{t}k_{j}\right)}
\]
\end{enumerate}
\end{prop}

\begin{proof}
(1) Denote $\Delta=\sum_{j=1}^{t}r-k_{j}$, so that $\left|M\right|=q^{rm+\Delta}$,
and there are $|M|-q^{(r-1)m+\Delta}=q^{(r-1)m+\Delta}\left(q^{m}-1\right)$
primitive vectors in $M$. More generally, if $v_{1},\ldots,v_{j}\in M$
are independent (where $0\leq j<m$), then $M/\left\langle v_{1},\ldots,v_{j}\right\rangle \cong\mathcal{O}_{r}^{m-j}\times\mathfrak{p}^{k_{1}}\times\ldots\times\mathfrak{p}^{k_{t}}$,
so there are $q^{(r-1)(m-j)+\Delta}\left(q^{m-j}-1\right)$ primitive
vectors in it, and each lifts to $|\left\langle v_{1},\ldots,v_{j}\right\rangle |=q^{rj}$
vectors in $M$. This gives all possible $v_{j+1}$ for which $v_{1},\ldots,v_{j+1}$
are independent, so by induction there are 
\[
\prod\nolimits_{j=0}^{n-1}q^{(r-1)(m-j)+\Delta+rj}\left(q^{m-j}-1\right)=q^{{n \choose 2}+n((r-1)m+\Delta)}\prod\nolimits_{i=m-n+1}^{m}\left(q^{i}-1\right)
\]
options for independent $v_{1},\ldots,v_{n}\in M$. Each spans an
$n$-space in $M$, and each $n$-space is obtained $\left|GL_{n}(\mathcal{O}_{r})\right|=q^{{n \choose 2}+n^{2}(r-1)}\prod_{i=1}^{n}\left(q^{i}-1\right)$
times, giving
\[
\frac{q^{{n \choose 2}+n((r-1)m+\Delta)}\prod_{i=m-n+1}^{m}\left(q^{i}-1\right)}{q^{{n \choose 2}+n^{2}(r-1)}\prod_{i=1}^{n}\left(q^{i}-1\right)}={\textstyle {m \choose n}_{q}\,}q^{(r-1)n(m-n)+n\Delta}=S_{n}^{m}q^{n\Delta}.
\]
For (2), note that $\varphi_{\bot}$ interchanges $n$-spaces containing
$M$ with $(d-n)$-spaces contained in $M^{\bot}\cong\mathcal{O}_{r}^{d-m-t}\times\mathfrak{p}^{r-k_{t}}\times\ldots\times\mathfrak{p}^{r-k_{1}}$,
of which there are $S_{d-n}^{d-m-t}q^{(d-n)\sum_{j=1}^{t}k_{j}}$
by (1).
\end{proof}

We now briefly explore the connection to the theory of buildings.
Let $G=GL_{d}(\mathcal{O}_{r})$, and let $\mathscr{B}=\left\{ L_{1},\ldots,L_{d}\right\} $
be the lines spanned by the standard basis of $\mathcal{O}_{r}^{d}$.
Let $\mathcal{A}\subseteq\mathcal{X}$ be the subcomplex induced by
modules spanned by subsets of $\mathscr{B}$, and $\sigma=\left\{ L_{1}+\ldots+L_{j}\,\middle|\,1\leq j\leq d-1\right\} $
(a $(d-2)$-dimensional facet in $\mathcal{A}$). As $\mathcal{X}$
is evidently covered by the $G$-translations of $\mathcal{A}$, it
is natural to ask whether $\left(\mathcal{X},G\mathcal{A}\right)$
is a Tits building. Equivalently, we can ask for a (B,N)-pair: The
$G$-stabilizer of $\sigma$ is the group of upper triangular matrices,
which we denote by $B$. The diagonal matrices $T\leq G$ are the
pointwise stabilizer of $\mathcal{A}$, and the monomial matrices
$N=N_{G}(T)$ are its set-wise stabilizer, so that $W=\nicefrac{N}{T}=\nicefrac{N}{B\cap N}$
acts on $\mathcal{A}$. The set $S$ of permutation matrices of the
form $\left(i\ i+1\right)$ reflects $\mathcal{A}$ along the faces
of $\sigma$, and $\left(W,S\right)$ is a Coxeter group acting simply-transitively
on the $(d-1)$-cells in $\mathcal{A}$. Nevertheless, we have:
\begin{prop}
\label{prop:not-building}For $d\geq3$ and $r\geq2$, $(B,N)$ is
not a Tits (B,N)-pair, and $\left(\mathcal{X},G\mathcal{A}\right)$
is not a building.
\end{prop}

\begin{proof}
The (B,N)-pair axioms require that $sBw\subseteq BswB\cup BwB$ for
every $s\in S$, $w\in W$. Assume $d=3$, and let $w=s=\left(\begin{smallmatrix} & 1\\
1\\
 &  & 1
\end{smallmatrix}\right)$, $b=\left(\begin{smallmatrix}1 & \pi\\
 & 1\\
 &  & 1
\end{smallmatrix}\right)$. We need $sbw=\left(\begin{smallmatrix}1\\
\pi & 1\\
 &  & 1
\end{smallmatrix}\right)\in B\cup BwB$, but all $g\in B$ have $g_{2,1}=0$, and all $g$ in $BwB=\left(\begin{smallmatrix}\mathcal{O}_{r}^{\times} & * & *\\
 & \mathcal{O}_{r}^{\times} & *\\
 &  & \mathcal{O}_{r}^{\times}
\end{smallmatrix}\right)\left(\begin{smallmatrix} & \mathcal{O}_{r}^{\times} & *\\
\mathcal{O}_{r}^{\times} & * & *\\
 &  & \mathcal{O}_{r}^{\times}
\end{smallmatrix}\right)$ have $g_{2,1}\in\mathcal{O}_{r}^{\times}$. \footnote{Note that the failure of $\mathcal{O}_{r}$ to be a field is precisely
what makes this example work.} The same example works for larger $d$, by placyin $s,w,b$ at the
top-left $3\times3$ block, and $I_{d-3}$ at the bottom-right block.

From the geometric perspective, $\left(\mathcal{X},G\mathcal{A}\right)$
is not a building as $\sigma=\left\{ \left\langle e_{1}\right\rangle ,\left\langle e_{1},e_{2}\right\rangle \right\} $
does not share an apartment with $sbw\sigma=\left\{ \left\langle e_{1}+\pi e_{2}\right\rangle ,\left\langle e_{1},e_{2}\right\rangle \right\} $
(for any $d\geq3$). Indeed, each apartment $g\mathcal{A}$ is induced
by the modules spanned by subsets of the basis $g\mathscr{B}$, and
no basis for $\mathcal{O}_{r}^{d}$ contains both $e_{1}$ and $e_{1}+\pi e_{2}$.
\end{proof}
For $d=2$, the complex $\mathcal{X}$ is a discrete set, which is
trivially a building if one considers all pairs of points as apartments.
However, if one takes only the ``algebraic'' apartments $G\mathcal{A}$,
the building axioms still fail by the same argument as in the proof
above.

\section{\protect\label{sec:Automorphic-ridigity}Automorphic rigidity}

Let $\mathcal{X}=\mathbb{P}_{fr}^{d-1}\left(\mathcal{O}_{r}\right)$
for $\mathcal{O},\mathfrak{p},r,d$ as before. The main goal of this
section is Theorem \ref{thm:aut0X}, which determines the automorphism
group of $\mathcal{X}$ when $d\geq3$ (note that for $d=2$, $\mathcal{X}$
is just a discrete set, so $\Aut(\mathcal{X})\cong S_{|\mathcal{X}|}$
and Theorem \ref{thm:aut0X} does not hold). The proof appears at
the end of the section, after building the necessary machinery. We
begin with a simple lemma:
\begin{lem}
\label{lem:trichotomy}For $V\leq\mathcal{O}_{r}^{d},n<d$, exactly
one of the following holds:
\begin{enumerate}
\item $V$ is not $n$-generated (and is not contained in any $n$-space).
\item $V$ is freely $n$-generated (and is contained in only one $n$-space).
\item $V$ is $n$-generated but not freely, and is contained in more than
one $n$-space.
\end{enumerate}
\end{lem}

\begin{proof}
If $V$ is not $n$-generated, it is not contained in any $n$-space
(see Fact \ref{fact:fact}). If $V$ is freely $n$-generated, the
only $n$-space containing it is itself. If $V$ is $n$-generated
but not freely, then it is of type $(m;k_{1},\ldots,k_{n-m})$ for
$m<n$, and we assume by applying some $g\in GL_{d}(\mathcal{O}_{r})$
that $V=\mathcal{O}_{r}^{m}\times\mathfrak{p}^{k_{1}}\times\ldots\times\mathfrak{p}^{k_{n-m}}\times0^{\times(d-n)}$.
As $k_{n-m}>0$, there exist more than one $n$-space containing $V$,
e.g.\ $\left\langle e_{1},\ldots,e_{n}\right\rangle $ and $\left\langle e_{1},\ldots,e_{n-1},e_{n}+\pi^{r-1}e_{n-1}\right\rangle $
(alternatively, by Proposition \ref{prop:contains_submodules}).
\end{proof}
Let us denote by $\overline{\mathcal{X}}$ the flag complex obtained
from all submodules of $\mathcal{O}_{r}^{d}$ including $0$ and $\mathcal{O}_{r}^{d}$,
which is topologically the suspension of $\mathcal{X}$. We observe
that any $\varphi\in\Aut^{0}(\mathcal{X})$ extends naturally to $\overline{\mathcal{X}}$
by $\varphi(0)=0$, $\varphi\left(\mathcal{O}_{r}^{d}\right)=\mathcal{O}_{r}^{d}$. 
\begin{lem}
\label{lem:subsetspan}Let $\varphi\in\Aut^{0}(\mathcal{X})$. If
$V\in\overline{\mathcal{X}}$ and $V=L_{1}+\ldots+L_{k}$ for $L_{i}\in\mathcal{X}_{1}$,
then $\varphi(V)=\varphi(L_{1})+\ldots+\varphi(L_{k})$.
\end{lem}

\begin{proof}
Let $m=\rank V$ ($0\leq m\leq k$). Since $V=L_{1}+\ldots+L_{k}$,
$V$ is the unique vertex in $\mathcal{X}_{m}$ which is adjacent
to all of $L_{1},\ldots,L_{k}$. It follows that $\varphi\left(V\right)$
is the unique vertex in $\mathcal{X}_{m}$ adjacent to $\varphi\left(L_{1}\right),\ldots,\varphi\left(L_{k}\right)$,
hence $\varphi\left(V\right)$ is the unique $m$-space containing
$W:=\varphi\left(L_{1}\right)+\ldots+\varphi\left(L_{k}\right)$.

If $m<d$ then Lemma \ref{lem:trichotomy} shows that $W\in\mathcal{X}_{m}$,
hence $W=\varphi(V)$ as desired. If $m=d$, then $V=\varphi(V)=\mathcal{O}^{d}$
and we need to show $W=\mathcal{O}^{d}$. We choose generating vectors,
$v_{i}$ for $L_{i}$ and $w_{i}$ for $\varphi(L_{i})$, and we can
assume now that $k=d$, throwing away redundant vectors using Lemma
\ref{lem:subsetspan}, so that $v_{1},\ldots,v_{d}$ is a basis. Assume
to the contrary that $W\lneq\mathcal{O}^{d}$, so that $w_{1},\ldots,w_{d}$
are linearly dependent, say $\pi^{j}w_{1}=\sum_{i=2}^{d}c_{i}w_{i}$
with $j<r$. Fixing a primitive $u\in\left\langle w_{2},\ldots,w_{d}\right\rangle $
such that $\sum_{i=2}^{d}c_{i}w_{i}\in\left\langle u\right\rangle $,
we see that $\left\langle w_{1}\right\rangle $ and $\left\langle u\right\rangle $
intersect nontrivially, so $w_{1}+u\notin\mathcal{X}_{2}$ and by
Lemma \ref{lem:trichotomy} there is more than one plane in $\mathcal{X}_{2}$
containing $w_{1},u$. Thus, there is more than one path $\mathcal{X}_{d-1}\rightarrow\mathcal{X}_{1}\rightarrow\mathcal{X}_{2}\rightarrow\mathcal{X}_{1}$
of the form $\left\langle w_{2},\ldots,w_{d}\right\rangle \rightarrow\left\langle u\right\rangle \rightarrow*\rightarrow\left\langle w_{1}\right\rangle $
(note $\left\langle w_{2},\ldots,w_{d}\right\rangle =\varphi(\left\langle v_{2},\ldots,v_{d}\right\rangle )\in\mathcal{X}_{d-1}$
by the first part of the proof). Applying $\varphi^{-1}$ to each
such path gives $\left\langle v_{2},\ldots,v_{d}\right\rangle \rightarrow\varphi^{-1}\left(\left\langle u\right\rangle \right)\rightarrow*\rightarrow\left\langle v_{1}\right\rangle $,
but there is only one path of this form since $v_{1},\ldots,v_{d}$
are independent, and we are done.
\end{proof}
Two immediate useful corollaries are:
\begin{cor}
\label{cor:Xn-determines}For $\varphi\in\Aut^{0}\left(\mathcal{X}\right)$,
\begin{enumerate}
\item $\varphi$ is determined by $\varphi\big|_{\mathcal{X}_{n}}$ for
any $1\leq n\leq d-1$.
\item If $v_{1},\ldots,v_{k}$ are linearly independent and $\varphi(\left\langle v_{j}\right\rangle )=\left\langle w_{j}\right\rangle $,
then $w_{1},\ldots,w_{k}$ are linearly independent.
\end{enumerate}
\end{cor}

\begin{proof}
(1) If follows from the theorem that $\varphi|_{\mathcal{X}_{1}}$
determines $\varphi$. In addition, $\varphi|_{\mathcal{X}_{n}}$
determines $\varphi|_{\mathcal{X}_{1}}$: $L\in\mathcal{X}_{1}$ is
the unique neighbor of all of $\left\{ V\in\mathcal{X}_{n}\,\middle|\,L\leq V\right\} $,
which implies that $\varphi\left(L\right)$ is the unique neighbor
of $\left\{ \varphi\left(V\right)\,\middle|\,L\leq V\in\mathcal{X}_{n}\right\} $.
(2) is immediate.
\end{proof}
From now we fix $\varphi\in\Aut^{0}(\mathcal{X})$, aiming to show
it is of the form $\varphi_{A,\tau}$ for some $A,\tau$. We denote
by $e_{1},\ldots,e_{d}$ the standard basis of $\mathcal{O}_{r}^{d}$.
\begin{lem}
\label{lem:wj-def}For $w_{1}\in\mathcal{O}_{r}^{d}$ such that $\varphi\left(\langle e_{1}\rangle\right)=\left\langle w_{1}\right\rangle $,
there exist unique $w_{2},\ldots,w_{d}$ such that $w_{1},\ldots,w_{d}$
is a basis satisfying 
\begin{equation}
\varphi\left(\left\langle e_{j}\right\rangle \right)=\left\langle w_{j}\right\rangle \text{ and }\varphi\left(\left\langle e_{1}+e_{j}\right\rangle \right)=\left\langle w_{1}+w_{j}\right\rangle \quad\left(2\leq j\leq d\right).\label{eq:e1+ej-w1+wj}
\end{equation}
\end{lem}

\begin{proof}
For $j\geq2$, let $u_{j}$ be a generator for $\varphi\left(\left\langle e_{j}\right\rangle \right)$.
We note that $e_{1}+e_{j}$ is primitive, and $\left\langle e_{1},e_{j}\right\rangle $
is the unique neighbor in $\mathcal{X}_{2}$ of any two among $\left\langle e_{1}\right\rangle ,\left\langle e_{j}\right\rangle $
and $\left\langle e_{1}+e_{j}\right\rangle $. By Lemma \ref{lem:subsetspan}
we have $\varphi(\left\langle e_{1},e_{j}\right\rangle )=\left\langle w_{1},u_{j}\right\rangle $,
which is thus the unique neighbor in $\mathcal{X}_{2}$ of any two
among $\left\langle w_{1}\right\rangle ,\left\langle u_{j}\right\rangle $,
and $\varphi(\left\langle e_{1}+e_{j}\right\rangle )$. It follows
that $\varphi\left(\langle e_{1}+e_{j}\rangle\right)=\langle\lambda w_{1}+\mu u_{j}\rangle$
for some scalars $\lambda,\mu$, and both of these are in $\mathcal{O}_{r}^{\times}$:
if $\mu\in\mathfrak{p}$ then $\lambda\in\mathcal{O}^{\times}$ by
primitivity of $\lambda w_{1}+\mu u_{j}$, and it follows that $0\neq\mathfrak{p}^{r-1}w_{1}\leq\left\langle \lambda w_{1}+\mu u_{j}\right\rangle \cap\left\langle w_{1}\right\rangle $.
This implies that $\varphi(\left\langle e_{1}+e_{j}\right\rangle )$
and $\left\langle w_{1}\right\rangle $ and have more than one neighbor
in $\mathcal{X}_{2}$, which is false, so that $\mu\in\mathcal{O}^{\times}$.
A similar argument shows that $\lambda\in\mathcal{O}^{\times}$. Thus,
$w_{j}=\frac{\mu}{\lambda}u_{j}$ satisfies (\ref{eq:e1+ej-w1+wj}),
and it is unique since $\left\langle w_{1}+\lambda w_{j}\right\rangle \neq\left\langle w_{1}+w_{j}\right\rangle $
for any $\lambda\neq1$. Finally, $w_{1},\ldots,w_{d}$ is a basis
by Lemma \ref{lem:subsetspan}.
\end{proof}
We fix the $w_{1},\ldots,w_{d}$ which were obtained in the last Lemma.
\begin{lem}
\label{lem:tj-def}For every $j$ there exists a unique permutation
$\tau_{j}$ on $\mathcal{O}_{r}$ such that $\varphi\left(\left\langle e_{1}+\xi e_{j}\right\rangle \right)=\left\langle w_{1}+\tau_{j}\left(\xi\right)w_{j}\right\rangle $
for any $\xi\in\mathcal{O}_{r}$.
\end{lem}

\begin{proof}
We have $\varphi\left(\left\langle e_{1}+\xi e_{j}\right\rangle \right)=\left\langle \lambda w_{1}+\mu w_{j}\right\rangle $
for some $\lambda,\mu\in\mathcal{O}_{r}$ by the same argument as
in the previous proof. Even if $\xi\notin\mathcal{O}_{r}^{\times}$,
we still have that $\left\langle e_{1},e_{j}\right\rangle $ is the
unique $\mathcal{X}_{2}$-neighbor of $\left\langle e_{1}+\xi e_{j}\right\rangle $
and $\left\langle e_{j}\right\rangle $, hence $\left\langle \lambda w_{1}+\mu w_{j}\right\rangle $
and $\left\langle w_{j}\right\rangle $ have a unique $\mathcal{X}_{2}$-neighbor,
which implies $\lambda\in\mathcal{O}_{r}^{\times}$ (but we may have
$\mu\in\mathfrak{p}$). To fulfill both $\varphi\left(\left\langle e_{1}+\xi e_{j}\right\rangle \right)=\left\langle w_{1}+\tau_{j}\left(\xi\right)w_{j}\right\rangle $
and (\ref{eq:e1+ej-w1+wj}) we must define $\tau_{j}(\xi)=\mu/\lambda$;
$\tau_{j}$ is injective, since $\xi'\neq\xi$ implies $\left\langle e_{1}+\xi e_{j}\right\rangle \neq\left\langle e_{1}+\xi'e_{j}\right\rangle $
which forces $\left\langle w_{1}+\tau_{j}\left(\xi\right)w_{j}\right\rangle \neq\left\langle w_{1}+\tau_{j}\left(\xi'\right)w_{j}\right\rangle $
by injectivity of $\varphi$. As $\mathcal{O}_{r}$ is finite, we
are done.
\end{proof}
We shall later see that $\tau_{j}$ are the same for all $j$, and
are in $\Aut_{Ring}\left(\mathcal{O}_{r}\right)$. It is already clear
from (\ref{eq:e1+ej-w1+wj}) that we have
\begin{equation}
\tau_{j}(0)=0\text{ and }\tau_{j}(1)=1\quad\left(2\leq j\leq d\right).\label{eq:t1t0}
\end{equation}

\begin{lem}
\label{lem:semi-linear}With the previous notations, we have
\begin{enumerate}
\item For any $\lambda_{2},\ldots,\lambda_{d}$, $\varphi\left(\langle\left(1,\lambda_{2},\ldots,\lambda_{d}\right)\rangle\right)=\left\langle w_{1}+\sum_{i=2}^{d}\tau_{i}(\lambda_{i})w_{i}\right\rangle $.
\item If some $\lambda_{i}$ is invertible, then $\varphi\left(\langle\left(0,\lambda_{2},\ldots,\lambda_{d}\right)\rangle\right)=\left\langle \sum_{i=2}^{d}\tau_{i}(\lambda_{i})w_{i}\right\rangle $
\end{enumerate}
\end{lem}

\begin{proof}
\emph{(1) }We proceed by induction on $k=\max\left\{ 1\leq i\leq d\,\middle|\,\lambda_{i}\neq0\right\} $
where $\lambda_{1}=1$; $k=1$ holds by definition, and $k=2$ by
Lemma \ref{lem:tj-def}. For general $k\geq3$, we write $\left(1,\lambda_{2},\ldots,\lambda_{d}\right)=\left(1,\lambda_{2},\ldots,\lambda_{k-1},0,\ldots,0\right)+\lambda_{k}e_{k}$,
and observe that $\varphi\left(\langle\left(1,\lambda_{2},\ldots,\lambda_{d}\right)\rangle\right)$
is contained in $\varphi\left(\langle\left(1,\lambda_{2},\ldots,\lambda_{k-1},0,\ldots,0\right),e_{k}\rangle\right)$,
which equals $\left\langle w_{1}+\sum_{i=2}^{k-1}\tau_{i}(\lambda_{i})w_{i},w_{k}\right\rangle $
by the induction hypothesis and Lemma \ref{lem:subsetspan}. Thus
we can write $\varphi\left(\langle\left(1,\lambda_{2},\ldots\lambda_{d}\right)\rangle\right)=\left\langle \mu_{1}\left(w_{1}+\sum_{i=2}^{k-1}\tau_{i}(\lambda_{i})w_{i}\right)+\mu_{2}w_{k}\right\rangle $
for some $\mu_{1},\mu_{2}\in\mathcal{O}_{r}$, and $\mu_{1}\in\mathcal{O}^{\times}$
as before, for otherwise $\left\langle \mu_{1}\left(w_{1}+\sum_{i=2}^{k-1}\tau_{i}(\lambda_{i})w_{i}\right)\right\rangle $
and $\left\langle w_{k}\right\rangle $ would have more than one neighbor
in $\mathcal{X}_{2}$, whereas $\langle\left(1,\lambda_{2},\ldots,\lambda_{k-1},0,\ldots,0\right)\rangle$
and $\left\langle e_{k}\right\rangle $ do not. Finally, we observe
that $\langle\left(1,\lambda_{2},\ldots\lambda_{d}\right)\rangle\leq\left\langle e_{1}+\lambda_{k}e_{k},e_{2},\ldots,e_{k-1}\right\rangle $
and applying $\varphi$ we obtain that $\left\langle w_{1}+\sum_{i=2}^{k-1}\tau_{i}(\lambda_{i})w_{i}+\frac{\mu_{2}}{\mu_{1}}w_{k}\right\rangle \leq\left\langle w_{1}+\tau_{k}\left(\lambda_{k}\right)w_{k},w_{2},\ldots,w_{k-1}\right\rangle $,
which holds only if $\mu_{2}/\mu_{1}=\tau_{k}(\lambda_{k})$ as claimed.

\emph{(2)} We have $\left\langle \left(0,\lambda_{2},\ldots,\lambda_{d}\right)\right\rangle \leq\left\langle e_{1},\left(1,\lambda_{2},\ldots,\lambda_{d}\right)\right\rangle $
so that $\varphi\left(\left\langle \left(0,\lambda_{2},\ldots,\lambda_{d}\right)\right\rangle \right)\leq\left\langle w_{1},w_{1}+\sum_{i=2}^{d}\tau_{i}(\lambda_{i})w_{i}\right\rangle $
(as always using Lemma \ref{lem:subsetspan}). In addition $\varphi\left(\left\langle \left(0,\lambda_{2},\ldots,\lambda_{d}\right)\right\rangle \right)\leq\left\langle w_{2},\ldots,w_{d}\right\rangle $,
and the intersection of the last two inclusions is precisely $\left\langle \sum_{i=2}^{d}\tau_{i}(\lambda_{i})w_{i}\right\rangle $.
\end{proof}
\begin{prop}
If $d\geq3$ then $\tau_{2},\ldots,\tau_{d}$ are identical ring automorphisms
of $\mathcal{O}_{r}$.
\end{prop}

\begin{proof}
We first prove that each $\tau_{j}$ is additive. For any $k\neq1,j$
and $\lambda,\mu\in\mathcal{O}$, we have $\left\langle e_{1}+(\mu+\lambda)e_{j}+e_{k}\right\rangle \leq\left\langle e_{1}+\mu e_{j},\lambda e_{j}+e_{k}\right\rangle $.
Applying $\varphi$ and using both parts of Lemma \ref{lem:semi-linear}
and (\ref{eq:t1t0}), we obtain $\left\langle w_{1}+\tau_{j}\left(\mu+\lambda\right)w_{j}+w_{k}\right\rangle \leq\left\langle w_{1}+\tau_{j}(\mu)w_{j},\tau_{j}(\lambda)w_{j}+w_{k}\right\rangle $,
which implies that $\tau_{j}\left(\mu+\lambda\right)=\tau_{j}(\mu)+\tau_{j}(\lambda)$.

In a similar manner, $\left\langle e_{1}+\mu\lambda e_{j}+\lambda e_{k}\right\rangle \leq\left\langle e_{1},\mu e_{j}+e_{k}\right\rangle $
gives $\left\langle w_{1}+\tau_{j}(\mu\lambda)w_{j}+\tau_{k}(\lambda)w_{k}\right\rangle \leq\left\langle w_{1},\tau_{j}(\mu)w_{j}+w_{k}\right\rangle $,
which implies that $\tau_{j}(\mu\lambda)=\tau_{k}(\lambda)\tau_{j}(\mu)$.
Taking $\mu=1$ we obtain that $\tau_{j}=\tau_{k}$, after which multiplicativity
follows.
\end{proof}
From this point we shall denote $\tau=\tau_{2}=\ldots=\tau_{d}$.
\begin{cor}
\label{cor:zero-inv}If $\lambda_{1}$ is either zero or invertible,
then $\varphi\left(\langle\left(\lambda_{1},\ldots,\lambda_{d}\right)\rangle\right)=\left\langle \sum_{i=1}^{d}\tau(\lambda_{i})w_{i}\right\rangle $
\end{cor}

\begin{proof}
Lemma \ref{lem:semi-linear} shows this for $\lambda_{1}=0$, and
for $\lambda_{1}$ invertible we have
\[
\varphi\left(\langle\left(\lambda_{1},\ldots,\lambda_{d}\right)\rangle\right)=\varphi\left(\langle\left(1,\tfrac{\lambda_{2}}{\lambda_{1}},\ldots,\tfrac{\lambda_{d}}{\lambda_{1}}\right)\rangle\right)=\left\langle w_{1}+\smash{\sum\nolimits_{i=2}^{d}}\tau\left(\tfrac{\lambda_{i}}{\lambda_{1}}\right)w_{i}\right\rangle =\left\langle \vphantom{\tfrac{\lambda_{i}}{\lambda_{1}}}\smash{\sum\nolimits_{i=1}^{d}}\tau(\lambda_{i})w_{i}\right\rangle .\qedhere
\]
\end{proof}
If $\mathcal{O}_{r}$ is a field then there is nothing left to do,
but we need to go further:
\begin{prop}
\label{prop:semi-linear-prim}If $d\geq3$ then for any primitive
$\left(\lambda_{1},\ldots,\lambda_{d}\right)$ we have $\varphi\left(\langle\left(\lambda_{1},\ldots,\lambda_{d}\right)\rangle\right)=\left\langle \sum_{i=1}^{d}\tau(\lambda_{i})w_{i}\right\rangle $.
\end{prop}

\begin{proof}
Let $j$ be an invertible coordinate in $\left(\lambda_{1},\ldots,\lambda_{d}\right)$.
Let us begin again from Lemma \ref{lem:wj-def} with $j$ replacing
$1$, namely, setting $\varphi\left(\langle e_{j}\rangle\right)=\left\langle w_{j}\right\rangle $
and finding $w'_{k}$ $\left(k\neq j\right)$ such that $\varphi\left(\left\langle e_{k}\right\rangle \right)=\left\langle w'_{k}\right\rangle \text{ and }\varphi\left(\left\langle e_{j}+e_{k}\right\rangle \right)=\left\langle w_{j}+w_{k}'\right\rangle $,
and $\tau'$ such that $\varphi\left(\left\langle e_{j}+\xi e_{k}\right\rangle \right)=\left\langle w_{j}+\tau'\left(\xi\right)w_{k}'\right\rangle $
for any $\xi\in\mathcal{O}_{r}$. Applying Corollary \ref{cor:zero-inv}
with respect to the $j$-th coordinate we obtain $\varphi\left(\langle\left(\lambda_{1},\ldots,\lambda_{d}\right)\rangle\right)=\left\langle \sum_{i=1}^{d}\tau'(\lambda_{i})w_{i}'\right\rangle $.
To conclude, $\left\langle w_{j}+w_{k}\right\rangle =\varphi\left(\left\langle e_{j}+e_{k}\right\rangle \right)=\left\langle w_{j}+w_{k}'\right\rangle $
implies that $w_{k}=w'_{k}$ for all $k$, and from $\left\langle w_{j}+\tau(\xi)w_{k}\right\rangle =\varphi\left(\left\langle e_{j}+\xi e_{k}\right\rangle \right)=\left\langle w_{j}+\tau'(\xi)w_{k}\right\rangle $
we obtain that $\tau'=\tau$.
\end{proof}
\begin{proof}[Proof of Theorem \ref{thm:aut0X}]
Let $A$ be the matrix whose $j$-th column is the $w_{j}$ constructed
in Lemma \ref{lem:wj-def}, and let $\tau$ be the automorphism defined
in Lemma \ref{lem:tj-def}. For any primitive $v=\left(\lambda_{1},\ldots,\lambda_{d}\right)$,
Proposition \ref{prop:semi-linear-prim} gives 
\[
{\textstyle \left\langle A\overline{\tau}(v)\right\rangle =\left\langle A\overline{\tau}\sum_{i=1}^{d}\lambda_{i}e_{i}\right\rangle =\left\langle A\sum_{i=1}^{d}\tau\left(\lambda_{i}\right)e_{i}\right\rangle =\left\langle \sum_{i=1}^{d}\tau\left(\lambda_{i}\right)w_{i}\right\rangle =\varphi\left(\left\langle v\right\rangle \right),}
\]
so that $\varphi\big|_{\mathcal{X}_{1}}=\varphi_{A,\tau}\big|_{\mathcal{X}_{1}}$,
and by Corollary \ref{cor:Xn-determines}(1) we have $\varphi=\varphi_{A,\tau}$.
This implies that the map $GL_{d}\left(\mathcal{O}_{r}\right)\rtimes\Aut_{Ring}(\mathcal{O}_{r})\rightarrow\Aut^{0}\left(\mathcal{X}\right)$
is onto, and it obviously factors through $PGL_{d}$. On the other
hand, if $\varphi_{A,\tau}=id$ then considering $\left\langle e_{j}\right\rangle =\left\langle A\overline{\tau}(e_{j})\right\rangle =\left\langle Ae_{j}\right\rangle $
shows that $A$ is diagonal, and considering $A(\sum_{i=1}^{d}e_{i})$
shows it is actually scalar. It then follows that $\tau=id$ by considering
$\left\langle \overline{\tau}(1,\alpha,0,\ldots,0)\right\rangle $
for $\alpha\in\mathcal{O}_{r}$, and we obtain (\ref{eq:Aut0}).

Now let $\varphi\in\Aut\left(\mathcal{X}\right)$, and $\sigma\in\mathcal{X}^{d-2}$
a maximal free flag. As $\varphi(\sigma)\in\mathcal{X}^{d-2}$ as
well, $\varphi$ must permute the colors of $\sigma$'s vertices.
In addition, for every $\sigma'\in\mathcal{X}^{d}$ there is a sequence
$\sigma=\sigma_{0},\sigma_{1},\ldots,\sigma_{\ell}=\sigma'$ such
that $\sigma_{i},\sigma_{i+1}$ agree on all vertices but one, which
forces $\varphi$ to induce the same permutation on the colors of
$\sigma_{i},\sigma_{i+1}$, and thus of $\sigma$ and $\sigma'$.
If $\varphi(\mathcal{X}_{i})=\mathcal{X}_{j}$ then we must have $\left|\mathcal{X}_{i}\right|=\left|\mathcal{X}_{j}\right|$,
and since $\left|\mathcal{X}_{n}\right|={d \choose n}_{q}\cdot q^{\left(r-1\right)n\left(d-n\right)}$
(Proposition \ref{prop:Xn}) and $q\nmid{d \choose n}_{q}$, this
forces $j\in\{i,d-i\}$ for $r\geq2$ (for $r=1$ this is a standard
result). If $\varphi(\mathcal{X}_{1})=\mathcal{X}_{1}$, then every
$V\in\mathcal{X}_{1}$ has $S_{n-1}^{d-1}={d-1 \choose n-1}_{q}\cdot q^{\left(r-1\right)\left(n-1\right)\left(d-n\right)}$
neighbors in $\mathcal{X}_{n}$ (Theorem \ref{thm:Spectrum}), so
that $\varphi$ cannot take $\mathcal{X}_{n}$ to $\mathcal{X}_{d-n}$
unless $n=\frac{d}{2}$, so $\varphi$ preserves colors. If $\varphi(\mathcal{X}_{1})=\mathcal{X}_{d-1}$,
then we observe that $\varphi_{\bot}\colon V\mapsto V^{\bot}$ interchanges
$\mathcal{X}_{m}$ and $\mathcal{X}_{d-m}$ (Fact \ref{fact:fact});
Thus, $\varphi\circ\varphi_{\bot}$ preserves $\mathcal{X}_{1}$ and
therefore all colors, and we obtain (\ref{eq:Aut}).
\end{proof}

\section{\protect\label{sec:Subgraph-rigidity}Subgraph rigidity}

For $\mathcal{X}=\mathbb{P}_{fr}^{d-1}(\mathcal{O}_{r})$ and $S\subseteq\left\{ 1,\ldots,d-1\right\} $
we denote by $\mathcal{X}_{S}$ the subcomplex of $\mathcal{X}$ induced
by all vertices with colors in $S$. Our goal in this section is to
show that $\mathcal{X}_{m,n}$ already determines $\mathcal{X}$ in
entirety for any $1\leq m<n\leq d$, and furthermore that $\mathcal{X}$
and $\mathcal{X}_{m,n}$ have (almost) the same automorphism group. 
\begin{thm}
\label{thm:subgraph_rigidity}Let $\mathcal{X}=\mathbb{P}_{fr}^{d-1}(\mathcal{O}_{r})$
as before, and let $\mathcal{X}_{m,n}$ be the subgraph induced by
the vertices of colors $m$ and $n$ for some $1\leq m<n\leq d-1$.
\begin{enumerate}
\item The complex $\mathcal{X}$ can be reconstructed from $\mathcal{X}_{m,n}$.
\item Every automorphism of $\mathcal{X}_{m,n}$ extends uniquely to an
automorphism of $\mathcal{X}$, giving
\[
\Aut^{0}\left(\mathcal{X}_{m,n}\right)=\Aut^{0}\left(\mathcal{X}\right),\text{ and }\Aut\left(\mathcal{X}_{m,n}\right)=\begin{cases}
\Aut\left(\mathcal{X}\right) & m+n=d\\
\Aut^{0}\left(\mathcal{X}\right) & else.
\end{cases}
\]
\end{enumerate}
\end{thm}

\begin{cor}
\label{cor:isom-graph}If $\left|\mathcal{O}/\mathfrak{p}\right|=\left|\mathcal{O}'/\mathfrak{p}'\right|$
and $\mathcal{X}=\mathbb{P}_{fr}^{d-1}(\mathcal{O}_{r})$ and $\mathcal{X}'=\mathbb{P}_{fr}^{d-1}(\mathcal{O}'_{r})$
satisfy $\mathcal{X}_{m,n}\cong\mathcal{X}_{m,n}^{'}$ for some $1\leq n<m\leq d-1$,
then $\mathcal{X}\cong\mathcal{X}'$.
\end{cor}

\begin{proof}[Proof of Theorem \ref{thm:subgraph_rigidity}]
Let us show first the uniqueness in claim (2). Denoting $\mathcal{G}=\begin{cases}
\Aut\left(\mathcal{X}\right) & n=d-m\\
\Aut^{0}\left(\mathcal{X}\right) & else.
\end{cases}$, there is a restriction map $\Phi\colon\mathcal{G}\rightarrow\Aut(\mathcal{X}_{m,n})$.
If $\Phi\left(\varphi\right)=\Phi\left(\varphi'\right)$ for $\varphi,\varphi'\in\mathcal{G}$
then either $\varphi,\varphi'$ both preserve or reverse coloring.
In the former case Corollary \ref{cor:Xn-determines}(1) shows that
$\varphi=\varphi'$, and in the latter the same follows by considering
$\varphi\circ\varphi_{\bot}$ and $\varphi'\circ\varphi_{\bot}$.
To show claim (2) we need to establish that $\Phi$ is onto. We show
this together with claim (1), by demonstrating that we can reconstruct
and extend automorphisms from $\mathcal{X}_{S}$ to $\mathcal{X}_{S\cup\left\{ c\right\} }$
for various $c$. We observe several cases separately:

\textbf{(A) }Reconstructing $\mathcal{X}_{1,\ldots,i+1,n}$ from $\mathcal{X}_{1,\ldots,i,n}$
($1\leq i<n$): Let $L\in\mathcal{X}_{1}$ and $V\in\mathcal{X}_{i}$.
If $L+V$ is not an $(i+1)$-space, then being $(i+1)$-generated
it is contained in more than one $\left(i+1\right)$-space (Lemma
\ref{lem:trichotomy}), and thus the number of $n$-spaces containing
$L,V$ is minimal when $L+V$ is an $(i+1)$-space.

Thus, for all $L,V$ with a minimal number of common neighbors in
$\mathcal{X}_{n}$ we add a vertex labeled $L+V$ to $\mathcal{X}_{i+1}$,
and connect it to all the aforementioned neighbors in $\mathcal{X}_{n}$,
to $L$, and to $V$ and its subspaces. Next, we glue together two
vertices $L+V$ and $L'+V'$ if they have the same set of neighbors
in $\mathcal{X}_{n}$, uniting their neighbor sets in lower dimensions.
We have reconstructed $\mathcal{X}_{i+1}$ properly as every $\left(i+1\right)$-space
is obtained uniquely in this manner.

Extending $\varphi\in\Aut^{0}(\mathcal{X}_{1,\ldots,i,n})$: we wish
to extend $\varphi$ by $\varphi(V+L):=\varphi(V)+\varphi(L)$, for
$L\in\mathcal{X}_{1}$ and $V\in\mathcal{X}_{i}$ such that $L+V\in\mathcal{X}_{i+1}$.
As we have seen,\textcolor{blue}{{} }$L+V\in\mathcal{X}_{i+1}$ implies
that $L,V$ have the minimal possible number of common $\mathcal{X}_{n}$-neighbors,
hence the same is true for $\varphi(L)$ and $\varphi(V)$, so $\varphi(V)+\varphi(L)\in\mathcal{X}_{i+1}$.
If $L+V=L'+V'$ then $L,V$ and $L',V'$ have the same common $\mathcal{X}_{n}$-neighbors,
which implies that $\varphi(L'),\varphi(V')$ have the same common
$\mathcal{X}_{n}$-neighbors as $\varphi(L),\varphi(V)$, which implies
that $\varphi(V')+\varphi(L')=\varphi(V)+\varphi(L)$, i.e.\ $\varphi$
is well defined. As $\mathcal{X}$ is a clique complex, it remains
to show that $\varphi$ is a graph automorphism, and noting that our
extension process commutes with taking inverses (of $\varphi$), it
is enough to show that $\varphi$ takes neighbors to neighbors. First,
if $V+L\sim U$ for $V+L\in\mathcal{X}_{i+1}$ and $U\in\mathcal{X}_{n}$,
then $V,L\sim U$, hence $\varphi(V),\varphi(L)\sim\varphi(U)$, so
that $\varphi(V+L)=\varphi(V)+\varphi(L)\sim\varphi(U)$. Next, if
$V\in\mathcal{X}_{i}$ and $W\in\mathcal{X}_{i+1}$ are neighbors
then there exist $L\in\mathcal{X}_{1}$ such that $W=V+L$, and then
$\varphi(V)\subseteq\varphi(V)+\varphi(L)=\varphi(W)$. This also
handles lower dimensions by transitivity: if $U\in\mathcal{X}_{j}$
for $j<i$ is a neighbor of $W\in\mathcal{X}_{i+1}$, then there exist
$V\in\mathcal{X}_{i}$ with $U\subseteq V\subseteq W$, and then $\varphi(U)\subseteq\varphi(V)\subseteq\varphi(W)$.

\textbf{(B)}{\bfseries\footnote{This is a special case of \textbf{(C)} but it is much simpler, so
we present it as a warm-up.}}\textbf{ }Reconstructing $\mathcal{X}_{1,n,n+1}$ from $\mathcal{X}_{1,n}$:
For $L\in\mathcal{X}_{1}$ and $V\in\mathcal{X}_{n}$ we observe the
set of 3-paths $V\rightarrow L'\rightarrow V'\rightarrow L$ with
$L'\in\mathcal{X}_{1},V'\in\mathcal{X}_{n}$. If $L+V$ is an $n+1$-space
(equivalently, $L\cap V=0$), then for every $L'\subseteq V$ we have
$L\cap L'=0$, which implies $L+L'\cong\mathcal{O}_{r}^{2}$, so there
are $S_{n-2}^{d-2}$ choices of $V'$ containing both $L$ and $L'$.
In contrast, if $L\cap V\neq0$ then there exist $\mathcal{X}_{1}\ni L'\subseteq V$
with $L\cap L'\cong\mathfrak{p}^{k}$ ($0\leq k<r$), hence $L+L'\cong\mathcal{O}_{r}\times\mathfrak{p}^{r-k}$,
and there are more than $S_{n-2}^{d-2}$ possibilities for $V'$ completing
$V\rightarrow L'\rightarrow\square\rightarrow L$: $S_{n-1}^{d-1}$
if $k=0$, and $S_{n-2}^{d-2}q^{(d-n)(r-k)}$ otherwise. Thus, we
can detect the pairs $L,V$ with $L\cap V=0$, and for each one we
add a vertex labeled $L+V$ to $\mathcal{X}_{n+1}$, and connect it
to $L$ and $V$. We can find all lines contained in $L+V$ as follows:
for every $n-1$ sublines $L_{1},\ldots,L_{n-1}$ of $V$ such that
$L,L_{1},\ldots,L_{n-1}$ are only contained in a single $n$-space
$V'\in\mathcal{X}_{n}$, we have $V'=L\oplus L_{1}\oplus\ldots\oplus L_{n-1}\subseteq V+L$
(see Lemma \ref{lem:trichotomy}), and we connect $L+V$ to all the
lines contained in $V'$. This covers all lines in $L+V$, as if $L'\subseteq L\oplus V\cong\mathcal{O}_{r}^{n+1}$\textcolor{red}{{}
}then there exist $L_{1},\ldots,L_{n-1}$ as above with $L'\leq\left\langle L,L_{1}\right\rangle \leq\left\langle L,L_{1},\ldots,L_{n-1}\right\rangle $.
We can now glue together $L+V$ and $L'+V'$ whenever they contain
the same lines, obtaining $\mathcal{X}_{n+1}$.

\textbf{(C) }Reconstructing $\mathcal{X}_{m,n,n+1}$ from $\mathcal{X}_{m,n}$
($m<n$): We note first that for $V_{1},V_{2}\in\mathcal{X}_{n}$,
$V_{1}+V_{2}$ is an $(n+1)$-space iff $V_{1}\cap V_{2}$ is an $(n-1)$-space.
Let $\mathfrak{m}_{V_{1},V_{2}}$ be the number common $\mathcal{X}_{m}$-neighbors
of $V_{1}$ and $V_{2}$. If $V_{1}\cap V_{2}$ is of type $\left(s\,;k_{1},\ldots,k_{t}\right)$
then Proposition \ref{prop:contains_submodules}(1) gives
\[
\mathfrak{m}_{V_{1},V_{2}}={\textstyle {s \choose m}_{q}\cdot q^{\left(\left(s-m\right)\left(r-1\right)+\left(\sum_{j=1}^{t}r-k_{j}\right)\right)m}}.
\]
When $V_{1}\cap V_{2}$ is an $(n-1)$-space we obtain 
\begin{equation}
{\textstyle \mathfrak{m}_{V_{1},V_{2}}={n-1 \choose m}_{q}\cdot q^{\left(n-1-m\right)\left(r-1\right)m}},\label{eq:V1V2_free}
\end{equation}
and only then: (\ref{eq:V1V2_free}) implies ${s \choose m}_{q}={n-1 \choose m}_{q}$
since $q\nmid{n \choose k}_{q}$ whenever ${n \choose k}_{q}\neq0$,
forcing $s=n-1$ and $t=0$. Consequently, for every $V_{1},V_{2}$
satisfying (\ref{eq:V1V2_free}) we add a vertex labeled $V_{1}+V_{2}$
to $\mathcal{X}_{n+1}$, and connect it to $V_{1},V_{2}$ and all
their $m$-subspaces. Now, for every collection of $m$-spaces $W_{1},\ldots,W_{r}$,
their sum is an $n$-space iff they have a unique common neighbor
in $\mathcal{X}_{n}$ (by Lemma \ref{lem:trichotomy} applied to $V=W_{1}+\ldots+W_{r}$).
Thus, if $W_{1},\ldots,W_{r}\in\mathcal{X}_{m}$ have a unique common
$\mathcal{X}_{n}$-neighbor $V$, and each $W_{i}$ is contained in
$V_{1}$ or in $V_{2}$, then $W_{1}+\ldots+W_{r}=V\subseteq V_{1}+V_{2}$,
and we connect $V_{1}+V_{2}$ to all $m$-subspaces of $V$. Showing
that we have found all $m$-subspaces of $V_{1}+V_{2}$ is somewhat
technical, so we prove this separately in Lemma \ref{lem:V1+V2=000020subspaces}
below, and we can now glue $V_{1}+V_{2}$ with $V'_{1}+V'_{2}$ whenever
they have the same set of $m$-subspaces.

Next, for $\varphi\in\Aut^{0}\left(\mathcal{X}_{S}\right)$ with $m,n\in S$,
$m<n$, and $n+1\notin S$, we wish to extend $\varphi$ to $\mathcal{X}_{S\cup\{n+1\}}$
by $\varphi(V_{1}+V_{2})=\varphi(V_{1})+\varphi(V_{2})$ whenever
$V_{1},V_{2}\in\mathcal{X}_{n}$ and $V_{1}+V_{2}\in\mathcal{X}_{n+1}$.
As $V_{1}+V_{2}\in\mathcal{X}_{n+1}$ iff $V_{1},V_{2}$ have (\ref{eq:V1V2_free})
many common $\mathcal{X}_{m}$-neighbors, this is preserved by $\varphi$
and is thus equivalent to $\varphi(V_{1})+\varphi(V_{2})\in\mathcal{X}_{n+1}$.
To show $\varphi$ is well-defined, define 
\[
\mathcal{F}_{V_{1},V_{2}}=\left\{ W\in\mathcal{X}_{m}\,\middle|\,{\text{there exist }W_{1},\ldots,W_{r}\in\mathcal{X}_{m},\text{ each}\text{ a neighbor of \ensuremath{V_{1}} or \ensuremath{V_{2}},}\atop \text{having have a unique \ensuremath{\mathcal{X}_{n}}-neighbor, which itself neighbors \ensuremath{W}.}}\right\} .
\]
As these are graph-theoretic conditions, we obtain $\mathcal{F}_{\varphi(V_{1}),\varphi(V_{2})}=\left\{ \varphi(W)\,\middle|\,W\in\mathcal{F}_{V_{1},V_{2}}\right\} $.
By Lemma \ref{lem:V1+V2=000020subspaces}, $\mathcal{F}_{V_{1},V_{2}}$
is precisely the set of $\mathcal{X}_{m}$-neighbors of $V_{1}+V_{2}$,
so that if $V_{1}+V_{2}=V_{1}'+V_{2}'$ then $\mathcal{F}_{V_{1},V_{2}}=\mathcal{F}_{V_{1}',V_{2}'}$,
and therefore $\mathcal{F}_{\varphi(V_{1}),\varphi(V_{2})}=\mathcal{F}_{\varphi(V_{1}'),\varphi(V_{2}')}$.
This implies that $\varphi(V_{1})+\varphi(V_{2})$ and $\varphi(V_{1}')+\varphi(V_{2}')$
have the same $\mathcal{X}_{m}$-neighbors, and are thus equal. Finally
we now that $\varphi$ takes neighbors to neighbors: for $\mathcal{X}_{n}\ni V\subseteq W\in\mathcal{X}_{n+1}$,
we can find $V'\in\mathcal{X}_{n}$ such that $V+V'=W$, and then
$\varphi(V)\subseteq\varphi(V)+\varphi(V')=\varphi(W)$. Transitivity
then shows that $\varphi$ also preserves neighboring between $\mathcal{X}_{n+1}$
and $\mathcal{X}_{j}$ with any $j<n$. For $j\geq n+2$, if $\mathcal{X}_{n+1}\ni V_{1}+V_{2}\subseteq V\in\mathcal{X}_{j}$
for $V_{1},V_{2}\in\mathcal{X}_{n}$ then $V_{1},V_{2}\subseteq V$
implies $\varphi(V_{1}),\varphi(V_{2})\subseteq\varphi(V)$ and thus
$\varphi(V_{1}+V_{2})=\varphi(V_{1})+\varphi(V_{2})\subseteq\varphi(V)$.

\textbf{(D) }Reconstructing $\mathcal{X}_{m-1,m,n}$ from $\mathcal{X}_{m,n}$
($m<n$): This is obtained from case \textbf{(C)} by observing that
$\varphi_{\bot}$ gives isomorphisms $\mathcal{X}_{m,n}\cong\mathcal{X}_{d-n,d-m}$
and $\mathcal{X}_{d-n,d-m,d-m+1}\cong\mathcal{X}_{m-1,m,n}$. Similarly,
for $S$ with $m,n\in S$ and $m-1\notin S$, we can extend $\varphi\in\Aut^{0}(\mathcal{X}_{S})$
to $\mathcal{X}_{\{m-1\}\cup S}$ by applying $\varphi_{\bot}$ and
using case \textbf{(C)}.

In conclusion, we can reconstruct $\mathcal{X}$ from $\mathcal{X}_{m,n}$
and extend any color-preserving automorphism by applying cases \textbf{(D)},
\textbf{(A)}, and then \textbf{(B)} (or \textbf{(C)}) each applied
the appropriate number of times. If $n=d+1-m$ then $\mathcal{X}_{m,n}$
has color-reversing automorphisms as well, but each one can be extended
to $\mathcal{X}$ by composing it with $\varphi_{\bot}$, extending
to $\mathcal{X}$ and composing back with $\varphi_{\bot}$.
\end{proof}
\begin{lem}
\label{lem:V1+V2=000020subspaces}For $m<n$, let $V_{1},V_{2}\in\mathcal{X}_{n}$
with $V_{1}+V_{2}\in\mathcal{X}_{n+1}$, and denote $\ensuremath{S=\left\{ m\text{-spaces contained in }V_{1}\text{ or in }V_{2}\right\} }$.
If $\mathcal{X}_{m}\ni W\leq V_{1}+V_{2}$, then $W$ is contained
in some $n$-space $V$ which is the sum of elements from $S$.
\end{lem}

\begin{proof}
We can assume $W\nleq V_{1,}V_{2}$, as if $W\leq V_{i}$ then we
can take $V=V_{i}$, which equals the sum of its $m$-subspaces. By
extending a basis of $V_{1}\cap V_{2}$ to $V_{1}$ and then to $V_{1}+V_{2}$,
we identify $V_{1}+V_{2}$ with $\mathcal{O}_{r}^{n+1}$ so that $V_{1}\cap V_{2}=\left\langle e_{1},\ldots,e_{n-1}\right\rangle $
and $V_{1}=\left\langle e_{1},\ldots,e_{n}\right\rangle $, which
forces $V_{2}=\left\langle e_{1},\ldots,e_{n-1},\alpha e_{n}+e_{n+1}\right\rangle $
for some $\alpha\in\mathcal{O}_{r}$. We note that the subgroup of
$GL_{n+1}(\mathcal{O}_{r}^{n+1})$ which preserves the flag $V_{1}\cap V_{2}\leq V_{1}\leq V_{1}+V_{2}$
is $G=\left(\begin{smallmatrix}GL_{n-1}(\mathcal{O}_{r}) & * & *\\
-\ 0\ - & \mathcal{O}_{r}^{\times} & *\\
-\ 0\ - & 0 & \mathcal{O}_{r}^{\times}
\end{smallmatrix}\right)$, so applying any $g\in G$ to $\mathcal{O}_{r}^{n+1}$ gives another
valid identification.

Let $A\in M_{m\times(n+1)}(\mathcal{O}_{r})$ be such that $W=\mathcal{O}_{r}^{m}A$.
Acting on $A$ from the left by $GL_{m}(\mathcal{O}_{r})$ does not
change $\mathcal{O}_{r}^{m}A$, and acting from the right by $g\in G$
is equivalent to changing the identification of $V_{1}+V_{2}$ with
$\mathcal{O}_{r}^{n+1}$ by $g$, so we allow both actions w.l.o.g..
These are enough to bring the leftmost $n-1$ columns of $A$ to SNF,
and add any column to ones on its right. We thus assume that $A=\left(B|v|w\right)$
with $B\in M_{m\times n-1}(\mathcal{O}_{r})$ in SNF, and if $B_{i,i}=1$
then $v_{i}=w_{i}=0$. We have $B\neq I$, for otherwise $v=w=0$,
which would imply $W\leq V_{1}$.

If $m\leq n-2$ then the rightmost column of $B$ is zero, so that
$W\leq e_{n-1}^{\bot}$. We can then take $V=e_{n-1}^{\bot}$, as
$e_{n-1}^{\bot}=\left\langle e_{1},\ldots,e_{n-2},e_{n},\alpha e_{n}+e_{n+1}\right\rangle $
is a sum of $m$-spaces in $S$. We thus assume $m=n-1$, and observe
that since $\mathrm{rank}\left(A\,\mathrm{mod}\,\mathfrak{p}\right)=\mathrm{rank}\left(\nicefrac{\mathcal{O}}{\mathfrak{p}}\otimes W\right)=m=n-1$,
we must have $\mathrm{rank}\left(B\,\mathrm{mod}\,\mathfrak{p}\right)\geq n-3$,
so $B=\mathrm{diag}\left(1,\ldots,1,*,*\right)$. We observe the $2\times4$
right-bottom block of $A=\left(B|v|w\right)$, and split into cases:
\begin{description}
\item [{$\left(\begin{smallmatrix}1 & 0\\
0 & \beta
\end{smallmatrix}\middle|\begin{smallmatrix}0\\
\gamma
\end{smallmatrix}\middle|\begin{smallmatrix}0\\
1
\end{smallmatrix}\right)$:}] $W=\left\langle e_{1},\ldots,e_{n-2},\beta e_{n-1}+\gamma e_{n}+e_{n+1}\right\rangle $.
Denoting by $\overline{v}$ a primitive multiple of $0\neq v\in\mathcal{O}_{r}^{n+1}$,
we take $V=\big\langle e_{1},\ldots,e_{n-2},\overline{\beta e_{n-1}+(\gamma-\alpha)e_{n}}\rangle+\langle e_{1},\ldots,e_{n-2},\alpha e_{n}+e_{n+1}\big\rangle$
($\beta e_{n-1}+(\gamma-\alpha)e_{n}$ is nonzero as $W\nleq V_{2}$)
-- this is a sum of $m$-spaces contained in $V_{1}$ and $V_{2}$.
\item [{$\left(\begin{smallmatrix}\beta & 0\\
0 & \gamma
\end{smallmatrix}\middle|\begin{smallmatrix}1\\
\delta
\end{smallmatrix}\middle|\begin{smallmatrix}0\\
1
\end{smallmatrix}\right)$:}] $W=\left\langle e_{1},\ldots,e_{n-3},\beta e_{n-2}+e_{n},\gamma e_{n-1}+\delta e_{n}+e_{n+1}\right\rangle $,
and we can take $V=\big\langle e_{1},\ldots,e_{n-2},e_{n}\rangle+\langle e_{1},\ldots,e_{n-2},\gamma e_{n-1}+\alpha e_{n}+e_{n+1}\big\rangle$.
\item [{$\left(\begin{smallmatrix}\beta & 0\\
0 & \gamma
\end{smallmatrix}\middle|\begin{smallmatrix}\delta\\
1
\end{smallmatrix}\middle|\begin{smallmatrix}1\\
0
\end{smallmatrix}\right)$:}] $W=\left\langle e_{1},\ldots,e_{n-3},\beta e_{n-2}+\delta e_{n}+e_{n+1},\gamma e_{n-1}+e_{n}\right\rangle $,
and we can take $V=\big\langle e_{1},\ldots,e_{n-3},e_{n-1},e_{n}\rangle+\langle e_{1},\ldots,e_{n-3},e_{n-1},\beta e_{n-2}+\alpha e_{n}+e_{n+1}\big\rangle$.\qedhere
\end{description}
\end{proof}

\section{\protect\label{sec:Spectrum-of-subgraphs}Spectrum of subgraphs}

In this Section we study the combinatorics of induced subgraphs of
the free projective space. We take $\mathcal{O},\mathfrak{p},q,r$
and $\mathcal{X}=\mathbb{P}_{fr}^{d-1}\left(\mathcal{O}_{r}\right)$
as in the previous sections, and observe the graph $\mathcal{X}_{1,n}$
for a fixed $2\leq n\leq d-1$. It is a bipartite graph with sides
of sizes $\left|\mathcal{X}_{1}\right|=S_{1}^{d}$ and $\left|\mathcal{X}_{n}\right|=S_{n}^{d}={d \choose n}_{q}\cdot q^{\left(r-1\right)n\left(d-n\right)}$
by Proposition \ref{prop:Xn}. The neighbors of any $L\in\mathcal{X}_{1}$
are the $n$-spaces containing it, hence by the correspondence theorem
$\deg\,(L)=S_{n-1}^{d-1}$, and the degree of any $V\in\mathcal{X}_{n}$
is $S_{1}^{n}$ by definition. The main goal of this section is the
following theorem and corollary:
\begin{thm}
\label{thm:Spectrum}For $\mathcal{X}=\mathbb{P}_{fr}^{d-1}\left(\mathcal{O}_{r}\right)$
with $\left|\mathcal{O}/\mathfrak{p}\right|=q$, the adjacency spectrum
of $\mathcal{X}_{1,n}$ is:\\
\begin{tabular}{|c|c|c|}
\hline 
 & Eigenvalue & Multiplicity\tabularnewline
\hline 
\hline 
 & $\pm\sqrt{{d-1 \choose n-1}_{q}\left[n\right]_{q}\,q^{\left(r-1\right)\left(\left(d-1\right)^{2}-\left(d-1\right)n+2n-2\right)}}$ & $1$\tabularnewline
\hline 
 & ${\textstyle \pm\sqrt{{d-2 \choose n-1}_{q}\,q^{\left(d-2\right)\left(r-1\right)\left(d-n\right)+r(n-1)}}}$ & $\left[d\right]_{q}-1$\tabularnewline
\hline 
${\scriptstyle 0\leq k\leq r-2}$ & $\pm\sqrt{{d-2 \choose n-1}_{q}\,q^{\left(d-2\right)\left(r-1\right)\left(d-n\right)}\cdot q^{\left(k+1\right)\left(n-1\right)}}$ & $q^{\left(r-2-k\right)\left(d-1\right)}\left(q^{d-1}-1\right)\left[d\right]_{q}$\tabularnewline
\hline 
 & $0$ & ${d \choose n}_{q}q^{\left(r-1\right)n\left(d-n\right)}\!-\![d]_{q}q^{\left(d-1\right)\left(r-1\right)}$\tabularnewline
\hline 
\end{tabular}

In particular, its spectral expansion (maximal normalized non-trivial
eigenvalue) is approximately $1/\sqrt{q^{n-1}}$.

\end{thm}

\begin{cor}
The graph $\mathcal{X}_{1,d-1}$ is $\left[d-1\right]_{q}q^{\left(d-2\right)\left(r-1\right)}$-regular
with $2\left[d\right]_{q}q^{\left(d-1\right)\left(r-1\right)}$ vertices,
non-trivial eigenvalues $\left\{ \!\sqrt{q^{\left(d-2\right)\left(k+r\right)}}\right\} _{k=0}^{r-1}$
and spectral expansion $\frac{\left(q-1\right)\sqrt{q^{d-2}}}{q^{\left(d-1\right)}-1}\approx\frac{1}{\sqrt{q^{d-2}}}$.
\end{cor}

The proof appears at the end of this section. Note that the case $d=3$
and $n=d-1=2$ is the main result of \cite{Kaufman2019Freeflagsover}.
However, the proof given there is highly convoluted, and the one we
give here for the general case is much simpler, due to a better understanding
of the complex. For a start, it is a standard exercise that the adjacency
spectrum of any bi-partite graph $G=\left(L\sqcup R,E\right)$ with
$\left|L\right|\leq\left|R\right|$ satisfies 
\begin{equation}
\Spec A_{G}=\left\{ \pm\sqrt{\lambda}\,\middle|\,\lambda\in\Spec\left(A_{G}^{2}\big|_{L}\right)\right\} \cup\left\{ 0\right\} ^{|R|-|L|},\label{eq:bipartite_spec}
\end{equation}
so for us it is enough to study $A_{\mathcal{X}_{1,n}}^{2}\big|_{\mathcal{X}_{1}}$,
which we denote by $\Gamma^{2}$. For a line $\widetilde{L}\in\mathcal{X}_{1}$,
let
\[
A_{i}^{\widetilde{L}}\coloneqq\left\{ L\in\mathcal{X}_{1}\colon\widetilde{L}\cap L\cong\mathfrak{p}^{i}\right\} .
\]
\vspace{-4ex}

\begin{claim}
\begin{enumerate}
\item For any $\widetilde{L}\in\mathcal{X}_{1}$ we have $\left|A_{i}^{\widetilde{L}}\right|=\begin{cases}
1 & i=0\\
q^{i\left(d-1\right)}-q^{\left(i-1\right)\left(d-1\right)} & 0<i<r\\
S_{1}^{d}-q^{\left(r-1\right)\left(d-1\right)} & i=r
\end{cases}$.
\item For any $L\in A_{i}^{\widetilde{L}}$, the entry $\Gamma_{L,\widetilde{L}}^{2}$
only depends on $i$, and equals
\[
\mathfrak{s}_{i}:=\begin{cases}
S_{n-1}^{d-1} & i=0\\
S_{n-2}^{d-2}\cdot q^{\left(r-i\right)\left(d-n\right)} & 0<i\leq r.
\end{cases}
\]
\end{enumerate}
\end{claim}

\begin{proof}
(1) We can assume w.l.o.g.\ that $\widetilde{L}=\left\langle e_{1}\right\rangle $.
For $0\leq i\leq r-1$, $L\in\bigcup_{j=0}^{i}A_{j}^{\widetilde{L}}$
iff $L=\left\langle v\right\rangle $ for a primitive $v\in\mathcal{O}_{r}^{d}$
with $v_{2},\ldots,v_{d}\in\mathfrak{p}^{r-i}$, so that $v_{1}\in\mathcal{O}_{r}^{\times}$.
By scaling we can assume $v_{1}=1$, giving $\left|\bigcup_{j=0}^{\smash{i}}A_{j}^{\smash{\widetilde{L}}}\right|=\left|\mathfrak{p}^{r-i}\right|^{d-1}=q^{i\left(d-1\right)}$,
from which we deduce $\left|A_{i}^{\smash{\widetilde{L}}}\right|$
inductively, and also $\left|A_{r}^{\smash{\widetilde{L}}}\right|=\left|\mathcal{X}_{1}\right|-\left|\bigcup_{j=0}^{r-1}A_{j}^{\smash{\widetilde{L}}}\right|=S_{1}^{d}-q^{\left(r-1\right)\left(d-1\right)}$.
(2) follows from Proposition\textcolor{red}{{} }\ref{prop:contains_submodules}(2),
as the common neighbors of $L$ and $L'$ are the $n$-spaces containing
$L+\widetilde{L}\cong\mathcal{O}_{r}\times\mathfrak{p}^{r-i}$ (the
dependence only on $i$ can also be seen by the transitivity of $\text{Stab}_{GL_{d}\left(\mathcal{O}_{r}\right)}\left(\left\langle e_{1}\right\rangle \right)\curvearrowright A_{i}^{e_{1}}$).
\end{proof}
\begin{lem}
\label{lem:triangle_ineq}Let $a,b,c$ be free lines in $\mathcal{O}_{r}^{d}$.
Then $\Gamma^{2}=A_{\mathcal{X}_{1,n}}^{2}\big|_{\mathcal{X}_{1}}$
satisfies
\[
\Gamma_{a,c}^{2}\geq\min\left\{ \Gamma_{a,b}^{2},\Gamma_{b,c}^{2}\right\} ,
\]
so that $\mathfrak{s}_{0}-\Gamma_{\square,\square}^{2}$ is an ultrametric.
In particular, if $\Gamma_{a,b}^{2}\neq\Gamma_{b,c}^{2}$ then $\Gamma_{a,c}^{2}=\min\left\{ \Gamma_{a,b}^{2},\Gamma_{b,c}^{2}\right\} $.
\end{lem}

\begin{proof}
We observe first that $\mathfrak{s}_{i}\geq\mathfrak{s}_{j}\iff i\leq j$.
Let $b\cap a\cong\mathfrak{p}^{i}$, $b\cap c\cong\mathfrak{p}^{j}$
and assume w.l.o.g. $i\leq j$. As $b\cong\mathcal{O}_{r}$, this
implies $b\cap c\leq b\cap a$ (Fact \ref{fact:fact}(1)). It follows
that $\mathcal{O}_{r}\cong a\geq a\cap c\geq b\cap c\cong\mathfrak{p}^{j}$,
so that $a\cap c\cong\mathfrak{p}^{k}$ for some $k\leq j$, as needed.
\end{proof}
\begin{prop}
\label{prop:Gamma2-recur}$\Gamma^{2}$ is obtained as $G_{0}$ in
the following recursive process: 
\begin{align*}
G_{r} & \coloneqq\left(\mathfrak{s}_{r}\right)\in M_{1}\left(\mathbb{Z}\right)\\
G_{r-1} & \coloneqq G_{r}\otimes J_{\left[d\right]_{q}}+\left(\mathfrak{s}_{r-1}-\mathfrak{s}_{r}\right)I_{[d]_{q}} &  & (J=\text{all one matrix})\\
G_{i} & \coloneqq G_{i+1}\otimes J_{q^{d-1}}+\left(\mathfrak{s}_{i}-\mathfrak{s}_{i+1}\right)I_{[d]_{q}q^{(r-i-1)(d-1)}} &  & \text{for \ensuremath{i=r-2,...,0}}.
\end{align*}
\end{prop}

\begin{proof}
We give a presentation of $\Gamma^{2}$, by induction/recursion. We
first choose an ordering $L_{1},\ldots,L_{S_{1}^{d}}$ for $\mathcal{X}_{1}$:
We begin with some fixed line $L_{1}=\widetilde{L}$, and continue
with $L_{2},\ldots,L_{q^{d-1}}$ covering $A_{1}^{\widetilde{L}}$
in some arbitrary order. Using this partial order we can already observe
that the top-left $q^{d-1}$ square in $\Gamma^{2}$ has $\mathfrak{s}_{0}$
on the diagonal and $\mathfrak{s}_{1}$ in the first line and column
except the first entry. Lemma \ref{lem:triangle_ineq} now implies
the rest of the block contains $\mathfrak{s}_{1}$ as well.

We continue in this manner: assume that for $2\leq t\leq r-1$ we
have defined $L_{1},\ldots,L_{q^{\left(t-1\right)\left(d-1\right)}}$
and that they cover $\bigcup_{i=0}^{t-1}A_{i}^{\widetilde{L}}$. We
now add $A_{t}^{\widetilde{L}}$ as $L_{q^{\left(t-1\right)\left(d-1\right)}+1},\ldots,L_{q^{t\left(d-1\right)}}$,
but now we need to be careful with their inner order. We split $A_{t}^{\widetilde{L}}$
into classes by $L\sim L'$ if $\left|L\cap L'\right|>q^{r-t}$, and
pick representatives $M_{2},\ldots,M_{q^{d-1}}$ for the classes.
For each $M_{j}$ we choose $g_{j}\in GL_{d}(\mathcal{O}_{r})$ with
$g_{j}\widetilde{L}=M_{j}$, and add $g_{j}L_{1},\ldots,g_{j}L_{q^{(t-1)(d-1)}}$
in the next places, starting from $L_{(j-1)q^{\left(t-1\right)\left(d-1\right)}+1}$.
Note if we take also $M_{1}=\widetilde{L}$ and $g_{1}=1$ then this
is consistent with the existing $L_{1},\ldots,L_{q^{\left(t-1\right)\left(d-1\right)}}$.
Since $G$ preserves the number of common neighbors, each diagonal
$q^{(t-1)(d-1)}$-block coming from some $M_{j}$ is identical to
the first block. Now we observe the top-left $q^{t(d-1)}$-block:
its first row and column have $\mathfrak{s}_{t}$ everywhere, save
for the first $q^{(t-1)(d-1)}$ entries. By Lemma \ref{lem:triangle_ineq},
the entries in the rest of the $q^{t(d-1)}$-block are at least $\mathfrak{s}_{t}$.
And if $L,L'$ do not belong to the same $q^{(t-1)(d-1)}$-diagonal
block, then they came from different $M_{j}$ classes, so $\left|L\cap L'\right|\leq q^{r-t}$,
so that $\Gamma_{L,L'}^{2}\leq\mathfrak{s}_{t}$ and thus equals $\mathfrak{s}_{t}$.

Finally, we add $A_{r}^{\widetilde{L}}$ as $L_{q^{(r-1)(d-1)}+1},\ldots,L_{S_{1}^{d}}$
in the same way. The only difference is that there are more equivalence
classes in $A_{r}^{\widetilde{L}}$, giving $M_{2},\ldots,M_{[d]_{q}}$
and not $q^{d-1}$ as before.

\begin{minipage}[t]{0.55\columnwidth}%
\ 

In total, the ordering we chose for $\mathcal{X}_{1}$ gives $\Gamma^{2}$
with $\mathfrak{s}_{0}$ along the diagonal, $\mathfrak{s}_{1}$ in
blocks of size $q^{d-1}$ around the diagonal, $\mathfrak{s}_{2}$
in blocks of size $q^{2\left(d-1\right)}$ around the $q^{d-1}$-blocks
of $\mathfrak{s}_{1}$'s and $\mathfrak{s}_{0}$'s, and so on (see
example on the right). Namely, 
\[
\begin{alignedat}{1}\Gamma_{L_{i},L_{j}}^{2} & =\mathfrak{s}_{\delta(i,j)},\qquad\text{where}\\
\delta(i,j) & =\min\left(\left\{ r\right\} \cup\left\{ k\,\middle|\,{\textstyle \left\lfloor \frac{i}{q^{(d-1)k}}\right\rfloor =\left\lfloor \frac{j}{q^{(d-1)k}}\right\rfloor }\right\} \right).
\end{alignedat}
\]
It is not hard to see that the recursive procedure in the Proposition
describes this matrix precisely. The advantage of the recursive presentation
is its usefulness for the proof of Theorem \ref{thm:Spectrum}.%
\end{minipage}\ \ \ %
\fcolorbox{black}{white}{\begin{minipage}[t]{0.4\columnwidth}%
\[
\scalebox{0.6}{\text{\ensuremath{\left(\begin{smallmatrix}6 & 2 & 2 & 2 & 1 & 1 & 1 & 1 & 1 & 1 & 1 & 1 & 1 & 1 & 1 & 1 & 1 & 1 & 1 & 1 & 1 & 1 & 1 & 1 & 1 & 1 & 1 & 1\\
2 & 6 & 2 & 2 & 1 & 1 & 1 & 1 & 1 & 1 & 1 & 1 & 1 & 1 & 1 & 1 & 1 & 1 & 1 & 1 & 1 & 1 & 1 & 1 & 1 & 1 & 1 & 1\\
2 & 2 & 6 & 2 & 1 & 1 & 1 & 1 & 1 & 1 & 1 & 1 & 1 & 1 & 1 & 1 & 1 & 1 & 1 & 1 & 1 & 1 & 1 & 1 & 1 & 1 & 1 & 1\\
2 & 2 & 2 & 6 & 1 & 1 & 1 & 1 & 1 & 1 & 1 & 1 & 1 & 1 & 1 & 1 & 1 & 1 & 1 & 1 & 1 & 1 & 1 & 1 & 1 & 1 & 1 & 1\\
1 & 1 & 1 & 1 & 6 & 2 & 2 & 2 & 1 & 1 & 1 & 1 & 1 & 1 & 1 & 1 & 1 & 1 & 1 & 1 & 1 & 1 & 1 & 1 & 1 & 1 & 1 & 1\\
1 & 1 & 1 & 1 & 2 & 6 & 2 & 2 & 1 & 1 & 1 & 1 & 1 & 1 & 1 & 1 & 1 & 1 & 1 & 1 & 1 & 1 & 1 & 1 & 1 & 1 & 1 & 1\\
1 & 1 & 1 & 1 & 2 & 2 & 6 & 2 & 1 & 1 & 1 & 1 & 1 & 1 & 1 & 1 & 1 & 1 & 1 & 1 & 1 & 1 & 1 & 1 & 1 & 1 & 1 & 1\\
1 & 1 & 1 & 1 & 2 & 2 & 2 & 6 & 1 & 1 & 1 & 1 & 1 & 1 & 1 & 1 & 1 & 1 & 1 & 1 & 1 & 1 & 1 & 1 & 1 & 1 & 1 & 1\\
1 & 1 & 1 & 1 & 1 & 1 & 1 & 1 & 6 & 2 & 2 & 2 & 1 & 1 & 1 & 1 & 1 & 1 & 1 & 1 & 1 & 1 & 1 & 1 & 1 & 1 & 1 & 1\\
1 & 1 & 1 & 1 & 1 & 1 & 1 & 1 & 2 & 6 & 2 & 2 & 1 & 1 & 1 & 1 & 1 & 1 & 1 & 1 & 1 & 1 & 1 & 1 & 1 & 1 & 1 & 1\\
1 & 1 & 1 & 1 & 1 & 1 & 1 & 1 & 2 & 2 & 6 & 2 & 1 & 1 & 1 & 1 & 1 & 1 & 1 & 1 & 1 & 1 & 1 & 1 & 1 & 1 & 1 & 1\\
1 & 1 & 1 & 1 & 1 & 1 & 1 & 1 & 2 & 2 & 2 & 6 & 1 & 1 & 1 & 1 & 1 & 1 & 1 & 1 & 1 & 1 & 1 & 1 & 1 & 1 & 1 & 1\\
1 & 1 & 1 & 1 & 1 & 1 & 1 & 1 & 1 & 1 & 1 & 1 & 6 & 2 & 2 & 2 & 1 & 1 & 1 & 1 & 1 & 1 & 1 & 1 & 1 & 1 & 1 & 1\\
1 & 1 & 1 & 1 & 1 & 1 & 1 & 1 & 1 & 1 & 1 & 1 & 2 & 6 & 2 & 2 & 1 & 1 & 1 & 1 & 1 & 1 & 1 & 1 & 1 & 1 & 1 & 1\\
1 & 1 & 1 & 1 & 1 & 1 & 1 & 1 & 1 & 1 & 1 & 1 & 2 & 2 & 6 & 2 & 1 & 1 & 1 & 1 & 1 & 1 & 1 & 1 & 1 & 1 & 1 & 1\\
1 & 1 & 1 & 1 & 1 & 1 & 1 & 1 & 1 & 1 & 1 & 1 & 2 & 2 & 2 & 6 & 1 & 1 & 1 & 1 & 1 & 1 & 1 & 1 & 1 & 1 & 1 & 1\\
1 & 1 & 1 & 1 & 1 & 1 & 1 & 1 & 1 & 1 & 1 & 1 & 1 & 1 & 1 & 1 & 6 & 2 & 2 & 2 & 1 & 1 & 1 & 1 & 1 & 1 & 1 & 1\\
1 & 1 & 1 & 1 & 1 & 1 & 1 & 1 & 1 & 1 & 1 & 1 & 1 & 1 & 1 & 1 & 2 & 6 & 2 & 2 & 1 & 1 & 1 & 1 & 1 & 1 & 1 & 1\\
1 & 1 & 1 & 1 & 1 & 1 & 1 & 1 & 1 & 1 & 1 & 1 & 1 & 1 & 1 & 1 & 2 & 2 & 6 & 2 & 1 & 1 & 1 & 1 & 1 & 1 & 1 & 1\\
1 & 1 & 1 & 1 & 1 & 1 & 1 & 1 & 1 & 1 & 1 & 1 & 1 & 1 & 1 & 1 & 2 & 2 & 2 & 6 & 1 & 1 & 1 & 1 & 1 & 1 & 1 & 1\\
1 & 1 & 1 & 1 & 1 & 1 & 1 & 1 & 1 & 1 & 1 & 1 & 1 & 1 & 1 & 1 & 1 & 1 & 1 & 1 & 6 & 2 & 2 & 2 & 1 & 1 & 1 & 1\\
1 & 1 & 1 & 1 & 1 & 1 & 1 & 1 & 1 & 1 & 1 & 1 & 1 & 1 & 1 & 1 & 1 & 1 & 1 & 1 & 2 & 6 & 2 & 2 & 1 & 1 & 1 & 1\\
1 & 1 & 1 & 1 & 1 & 1 & 1 & 1 & 1 & 1 & 1 & 1 & 1 & 1 & 1 & 1 & 1 & 1 & 1 & 1 & 2 & 2 & 6 & 2 & 1 & 1 & 1 & 1\\
1 & 1 & 1 & 1 & 1 & 1 & 1 & 1 & 1 & 1 & 1 & 1 & 1 & 1 & 1 & 1 & 1 & 1 & 1 & 1 & 2 & 2 & 2 & 6 & 1 & 1 & 1 & 1\\
1 & 1 & 1 & 1 & 1 & 1 & 1 & 1 & 1 & 1 & 1 & 1 & 1 & 1 & 1 & 1 & 1 & 1 & 1 & 1 & 1 & 1 & 1 & 1 & 6 & 2 & 2 & 2\\
1 & 1 & 1 & 1 & 1 & 1 & 1 & 1 & 1 & 1 & 1 & 1 & 1 & 1 & 1 & 1 & 1 & 1 & 1 & 1 & 1 & 1 & 1 & 1 & 2 & 6 & 2 & 2\\
1 & 1 & 1 & 1 & 1 & 1 & 1 & 1 & 1 & 1 & 1 & 1 & 1 & 1 & 1 & 1 & 1 & 1 & 1 & 1 & 1 & 1 & 1 & 1 & 2 & 2 & 6 & 2\\
1 & 1 & 1 & 1 & 1 & 1 & 1 & 1 & 1 & 1 & 1 & 1 & 1 & 1 & 1 & 1 & 1 & 1 & 1 & 1 & 1 & 1 & 1 & 1 & 2 & 2 & 2 & 6
\end{smallmatrix}\right)}}}
\]
Example of $\Gamma^{2}$ with $d=3,n=2$, $q=r=2$ for which $\mathfrak{s}_{0}=6,\mathfrak{s}_{1}=2,\mathfrak{s}_{2}=1$.\\
\end{minipage}}
\end{proof}
We can now finish the main proof:
\begin{proof}[Proof of Theorem \ref{thm:Spectrum}]
We can compute the spectra of $G_{i}$ from Proposition \ref{prop:Gamma2-recur}
recursively, using $\Spec J_{n}=\{n,0^{\times(n-1)}\}$, and multiplicity
of eigenvalues under $\otimes$. For $G_{0}$ we obtain:

\hfill{}%
\begin{tabular}{|c|c|c|}
\hline 
 & Eigenvalue of $G_{0}=\Gamma^{2}=A_{\mathcal{X}_{1,n}}^{2}\big|_{\mathcal{X}_{1}}$ & Multiplicity\tabularnewline
\hline 
\hline 
 & $q^{\left(r-1\right)\left(d-1\right)}\left[d\right]_{q}\mathfrak{s}_{r}+\sum_{i=0}^{r-1}q^{i\left(d-1\right)}\left(\mathfrak{s}_{i}-\mathfrak{s}_{i+1}\right)$ & $1$\tabularnewline
\hline 
 & $\sum_{i=0}^{r-1}q^{i\left(d-1\right)}\left(\mathfrak{s}_{i}-\mathfrak{s}_{i+1}\right)$ & $\left[d\right]_{q}-1$\tabularnewline
\hline 
${\scriptstyle 0\leq k\leq r-2}$ & $\sum_{i=0}^{k}q^{i\left(d-1\right)}\left(\mathfrak{s}_{i}-\mathfrak{s}_{i+1}\right)$ & $q^{\left(r-2-k\right)\left(d-1\right)}\left(q^{d-1}-1\right)\left[d\right]_{q}$\tabularnewline
\hline 
\end{tabular}\hfill{}

As $\left|\mathcal{X}_{1}\right|\leq\left|\mathcal{X}_{n}\right|$,
the spectrum of $A_{\mathcal{X}_{1,n}}$ is then obtained by (\ref{eq:bipartite_spec}),
substituting the values of $\mathfrak{s}_{i}$ and simplifying. Finally,
the spectral expansion of $\mathcal{X}_{1,n}$ is 
\[
\frac{\sqrt{{d-2 \choose n-1}_{q}\,q^{\left(d-2\right)\left(r-1\right)\left(d-n\right)+r(n-1)}}}{\sqrt{{d-1 \choose n-1}_{q}\left[n\right]_{q}\,q^{\left(r-1\right)\left(\left(d-1\right)^{2}-\left(d-1\right)n+2n-2\right)}}}=\sqrt{\frac{\left[d-n\right]_{q}}{\left[n\right]_{q}\left[d-1\right]_{q}}q^{n-1}}\approx\frac{1}{\sqrt{q^{n-1}}}.\qedhere
\]
\end{proof}

\section{\protect\label{sec:Isospectral-non-isopmorphic-grap}Isospectral
non-isomorphic graphs}
\begin{thm}
\label{thm:iso-noniso}If $p\in\mathbb{Z}$ is a prime, then for any
$r\geq2$, $d\geq3$ and $1<n<d$ the graphs 
\[
\mathbb{P}_{1,n}^{d-1}\left(\mathbb{Z}/(p^{r})\right),\quad\text{and }\quad\mathbb{P}_{1,n}^{d-1}\left(\mathbb{F}_{p}[t]/(t^{r})\right)
\]
are isospectral and non-isomorphic.
\end{thm}

\begin{proof}
Let $\mathcal{X}=\mathbb{P}_{fr}^{d-1}(\mathbb{Z}/p^{r})$ and $\mathcal{X}'=\mathbb{P}_{fr}^{d-1}(\mathbb{F}_{p}[t]/t^{r})$,
so that the graphs under questions are $\mathcal{X}_{1,n}$ and $\mathcal{X}'_{1,n}$,
which are isospectral by Theorem \ref{thm:Spectrum}. If $\mathcal{X}_{1,n}\cong\mathcal{X}'_{1,n}$
then $\Aut^{0}\left(\mathcal{X}_{1,n}\right)\cong\Aut^{0}\left(\mathcal{X}'_{1,n}\right)$,
so by Theorem \ref{thm:subgraph_rigidity}(2) $\Aut^{0}\left(\mathcal{X}\right)\cong\Aut^{0}\left(\mathcal{X}'\right)$.
Theorem \ref{thm:aut0X} now implies that 
\[
\mathcal{G}:=PGL_{d}\left(\mathbb{Z}/p^{r}\right)\rtimes\Aut_{Ring}(\mathbb{Z}/p^{r})\cong PGL_{d}\left(\mathbb{F}_{p}[t]/t^{r}\right)\rtimes\Aut_{Ring}(\mathbb{F}_{p}[t]/t^{r})=:\mathcal{G}'.
\]

In general for prime $\mathfrak{p}\trianglelefteq\mathcal{O}$ and
$\mathcal{O}/\mathfrak{p}\cong\mathbb{F}_{q}$ we have 
\[
\left|PGL_{d}\left(\mathcal{\mathcal{O}}/\mathfrak{p}^{r}\right)\right|=q^{d^{2}(r-1)}\left|PGL_{d}(\mathbb{F}_{q})\right|,
\]
and we focus on the ring automorphisms, of which $\mathbb{Z}/p^{r}$
has none. In contrast, endomorphisms of $\mathbb{F}_{p}[t]/t^{r}$
correspond to $sub_{f}\colon t\mapsto f(t)$ for $f\in t\cdot\mathbb{F}_{p}[t]$,
and $sub_{f}$ is an automorphism if $t^{2}\nmid f(t)$, since then
$f(t)^{r-1}\neq0$ shows that $t^{r-1}\notin\ker sub_{f}$. Thus,
\[
\left|\Aut_{Ring}\left(\mathbb{F}_{p}[t]/t^{r}\right)\right|=(p-1)p^{r-2}
\]
and unless $p=r=2$ we have $\left|\Aut_{Ring}\left(\mathbb{F}_{p}[t]/t^{r}\right)\right|\neq\left|\Aut_{Ring}\left(\mathbb{Z}/p^{r}\right)\right|=1$,
so that $\left|\mathcal{G}\right|\neq\left|\mathcal{G}'\right|$.

It is left to handle the case $p=r=2$, in which $\Aut_{Ring}\left(\mathbb{F}_{p}[t]/t^{r}\right)=\left\{ id\right\} $
and $\left|\mathcal{G}\right|=\left|\mathcal{G}'\right|$. The case
$d\geq4$ is handled by Theorem 10 of \cite{Bunina2019Isomorphismselementaryequivalence},
which shows that $PGL_{d}(R)$ determines $R$ for a local ring $R$,
if $d\geq4$, or $d\geq3$ and $2\in R^{\times}$. We are still left
with $d=3$, as $2$ is not invertible in $\mathbb{Z}/2^{r}$ and
$\mathbb{F}_{2}[t]/t^{r}$. Indeed, we find that there is an exceptional
isomorphism
\begin{gather*}
\Psi\colon GL_{3}\left(\mathbb{Z}/4\right)\overset{\cong}{\longrightarrow}GL_{3}\left(\mathbb{F}_{2}[t]/t^{2}\right)\text{ defined by}\\
\Psi\left(\begin{smallmatrix}0 & 1 & 1\\
0 & 0 & 1\\
1 & 0 & 0
\end{smallmatrix}\right)=\left(\begin{smallmatrix}1 & 1 & t+1\\
t & 1 & 1\\
1 & t & t+1
\end{smallmatrix}\right),\quad\Psi\left(\begin{smallmatrix}0 & 1 & 0\\
0 & 0 & 1\\
3 & 0 & 1
\end{smallmatrix}\right)=\left(\begin{smallmatrix}1 & 1 & t+1\\
1 & t & 1\\
t & 1 & 1
\end{smallmatrix}\right),
\end{gather*}
which shows that the assumptions in \cite{Bunina2019Isomorphismselementaryequivalence}
are truly necessary. This shows that $\mathcal{X}$ and $\mathcal{X}'$
actually have the same automorphism group when $d-1=p=r=2$. Nevertheless,
an explicit realization of these graphs in sage reveals them to be
non-isomorphic, settling the remaining case. It turns out that they
can be constructed as intimate Cayley graphs:
\begin{align*}
\mathcal{X} & =Cay\left(D_{7}\times C_{2}\times C_{2},\left\{ (\tau,0,0),(\tau,1,0),(\tau\sigma,1,0),(\tau\sigma,1,1),(\tau\sigma^{3},0,0),(\tau\sigma^{3},1,1)\right\} \right)\\
\mathcal{X}' & =Cay\left(D_{7}\times C_{2}\times C_{2},\left\{ (\tau,0,0),(\tau,1,0),(\tau\sigma,1,0),(\tau\sigma,1,1),(\tau\sigma^{3},1,0),(\tau\sigma^{3},0,1)\right\} \right),
\end{align*}
where $D_{7}=\left\langle \sigma,\tau\,\middle|\,\sigma^{7},\tau^{2},(\sigma\tau)^{2}\right\rangle $.
We do not know if this phenomenon holds in greater generality.
\end{proof}
If we look at $\mathcal{X}_{m,n}$ more generally, we can still prove
that $\mathbb{P}_{m,n}^{d-1}\left(\mathbb{Z}/(p^{r})\right)$ and
$\mathbb{P}_{m,n}^{d-1}\left(\mathbb{F}_{p}[t]/(t^{r})\right)$ are
not isomorphic by the same argument, but currently we do not know
whether they are isospectral. 
\begin{conjecture}
For $p,r,d$ as above and $2\leq m<n\leq d$ the graphs $\mathbb{P}_{m,n}^{d-1}\left(\mathbb{Z}/(p^{r})\right)$
and $\mathbb{P}_{m,n}^{d-1}\left(\mathbb{F}_{p}[t]/(t^{r})\right)$
are isospectral and non-isomorphic.
\end{conjecture}

\bibliographystyle{plain}
\bibliography{mybib}

\end{document}